\documentclass[11pt]{amsart}

\usepackage{tikz-cd} 
\usepackage[english]{babel}
\usepackage{amsmath}
\usepackage{amstext}
\usepackage{amsfonts}
\usepackage{amssymb}
\usepackage{amsthm}
\usepackage{amsrefs}
\usepackage{mathtools}
\usepackage{dsfont}
\usepackage{color}
\usepackage{graphicx}
\usepackage[colorinlistoftodos]{todonotes}
\usepackage{float}
\usepackage[colorinlistoftodos]{todonotes}
\usepackage{theoremref}
\usepackage{enumitem}
\setlist[enumerate]{itemsep=0mm}
\usepackage{dirtytalk}
\usepackage{todonotes}
\usepackage{marginnote} 

\usepackage{multicol} 
\usepackage{titletoc}
\allowdisplaybreaks
\let\oldaddcontentsline\addcontentsline

\newcommand{\starttocentries}{\let\addcontentsline\oldaddcontentsline}
\usepackage{enumerate}
\usepackage{enumitem}
\usepackage[active]{srcltx}
\usepackage{nicefrac}
\usepackage{amsthm,amssymb,setspace}
\usepackage{amsrefs}
\usepackage[matrix, arrow]{xy}
\usepackage[normalem]{ulem}
\usepackage{a4wide}
\usepackage{bbm}
\usepackage{comment}

\usepackage{makecell}

\newcommand{\id}{{\rm id}}
\newcommand{\N}{\mathbb N}
\newcommand{\Z}{\mathbb Z}

\newcommand{\C}{\mathbb C}

\newcommand{\suppe}{{\rm supp}}
\newcommand{\suppg}{{\it supp}}
\newcommand{\fpc}{{\rm fpc}}
\newcommand{\fix}{{\rm Fix}}

\newcommand{\ie}{\text{i.e.\;}}
\newcommand{\aut}{{\rm Aut}}
\newcommand{\fin}{{\rm fin}}

\newcommand{\sign}{\mathrm{sign}}
\newcommand{\Part}{\mathrm{Part}}
\newcommand{\Lin}{\mathrm{Lin}}
\newcommand{\Aut}{\textup{Aut}}

\theoremstyle{plain}
\newtheorem{theorem}{Theorem}[section]
\newtheorem{corollary}[theorem]{Corollary}
\newtheorem{lemma}[theorem]{Lemma}
\newtheorem{proposition}[theorem]{Proposition}

\theoremstyle{definition}
\newtheorem{question}[theorem]{Question}
\newtheorem{definition}[theorem]{Definition}

\newtheorem{thmx}{Theorem}

\theoremstyle{remark}

\newtheorem{remark}[theorem]{Remark}
\newtheorem{example}[theorem]{Example}

\numberwithin{equation}{section}

\begin{document}


\title[ISR for semidirect products]{Invariant subalgebras rigidity for von Neumann algebras of groups arising as certain semidirect products}

\author{Tattwamasi Amrutam}
\address{Institute of Mathematics of the Polish Academy of Sciences, ul.~\'Sniadeckich 8, 00--656 Warszawa, Poland}
\email{tattwamasiamrutam@impan.pl}

\author{Artem Dudko}
\address{Institute of Mathematics of the Polish Academy of Sciences, ul.~\'Sniadeckich 8, 00--656 Warszawa, Poland}
\email{adudko@impan.pl}

\author{Yongle Jiang}
\address{School of Mathematical Sciences, Dalian University of Technology, Dalian, 116024, China}
\email{yonglejiang@dlut.edu.cn}

\author{Adam Skalski}
\address{Institute of Mathematics of the Polish Academy of Sciences, ul.~\'Sniadeckich 8, 00--656 Warszawa, Poland}
\email{a.skalski@impan.pl}


\begin{abstract}
We study the ISR (von Neumann invariant subalgebra rigidity) property for certain discrete groups arising as semidirect products from algebraic actions on certain 2-torsion groups, mostly arising as direct products of $\Z_2$.  We present, in particular, the first example of an amenable group with the ISR property that admits a non-trivial abelian normal subgroup. Several other examples are discussed, notably including an infinite amenable group whose von Neumann algebra admits precisely one invariant von Neumann subalgebra which does not come from a normal subgroup. We also investigate the form of invariant subalgebras of the group von Neumann algebra of the standard lamplighter group.
\end{abstract}

\maketitle

\section{Introduction}

The study of von Neumann algebras of discrete groups dates back to the very origins of operator algebra theory, as these  appeared already in the foundational work of Murray and von Neumann. It was understood even then that the group von Neumann algebra \( L(G) \) of a discrete group \( G \) is a factor---i.e., it contains no nontrivial weakly closed ideals---if and only if \( G \) is ICC (i.e., all nontrivial conjugacy classes are infinite). A von Neumann ideal in \( L(G) \) is automatically an \emph{invariant} subalgebra, in the sense that it is preserved under the natural adjoint action of \( G \). In recent years, beginning with the work of Kalantar and Panagopoulos~\cite{kalantar2022invariant} and its predecessor due to Alekseev and Brugger~\cite{alekseev2019rigidity}, there has been renewed interest in the study of general invariant subalgebras. Notably, maximal abelian invariant subalgebras of \( L(G) \) serve as prototypical examples of \emph{Cartan subalgebras}, which have long played a central role in the theory of von Neumann algebras (see, e.g.,~\cites{feldman1977ergodic,popa1983maximal}).

It is easy to see that every normal subgroup \( N \trianglelefteq G \) yields an invariant subalgebra \( L(N) \subseteq L(G) \); these should be viewed as the \say{obvious} ones. In \cite{kalantar2022invariant}, Kalantar and Panagopoulos proved that if \( G \) is a lattice in a higher-rank Lie group, then all invariant subalgebras of \( L(G) \) arise from normal subgroups. This was taken up in \cite{amrutam2023invariant}, where the authors formally introduced the \emph{ISR} (invariant subalgebra rigidity) property for a discrete countable group \( G \), understood as the non-existence of \say{non-obvious} (or \say{exotic}) invariant subalgebras of \( L(G) \). They also showed that noncommutative free groups, and more generally certain \say{negatively curved} groups and their finite direct products, have the ISR property. The results of \cite{amrutam2023invariant} were later partially generalized in \cite{chifan2022invariant}.

It is not difficult to see that an infinite group with the ISR property must be ICC (\cite[Proposition~3.1]{amrutam2023invariant}). However, in contrast to factoriality, it does not appear to be easy to characterize the ISR property in simple group-theoretic terms. The examples given in \cite{kalantar2022invariant}, \cite{amrutam2023invariant}, and \cite{chifan2022invariant} mentioned above were all non-amenable; in fact, the first example of an infinite amenable group with the ISR property—namely \( S_\infty \)—was pointed out in \cite{jz}. The proof of the main result in \cite{jz} relied on detailed knowledge of the indecomposable characters of \( S_\infty \). After that, the so-called character approach was applied in \cite{dudko2024character} to prove, among other things, the ISR property for a broad class of approximately finite groups arising as symmetry groups of simple Bratteli diagrams. 

In fact, the only apparent obstruction to the ISR property for amenable groups, noted in \cite[Example~3.5]{amrutam2023invariant}, has been the existence of non-trivial abelian normal subgroups. The construction in \cite[Example~3.5]{amrutam2023invariant} shows that if \( G \) contains a normal abelian subgroup with elements of order at least 3, then \( G \) cannot have the ISR property. This observation serves as the motivation for the study undertaken in this paper, where we construct the first example of an amenable group \( G \) with a non-trivial normal abelian subgroup such that \( G \) possesses the ISR property. By the above discussion, the abelian part of such a group must consist entirely of elements of order $2$—that is, it must be a $2$-torsion group. It is then natural to consider wreath-type actions on infinite direct products of 2-torsion groups. A modification of this idea leads to the following result; we refer to the beginning of Section \ref{sec:ISR} for the detailed definition of the group appearing in the statement.

\begin{thmx} \label{thm:A}
 The group $S(2^\infty) \ltimes \tilde{C}(X; \mathbb Z_2)$ (which is amenable, ICC, and admits a non-trivial abelian normal subgroup) has the ISR property.
	\end{thmx}

In view of the above, it is natural to ask what happens for semidirect products involving actions on $\Z_2^\infty:= \bigoplus_{\N} \Z_2$. It turns out that we need not obtain the ISR property, but may witness a different interesting behaviour, as exemplified by the next statement, where $\Aut_{\textup{fin}}(\Z_2^\infty)$ denotes the group of all \emph{finitary} automorphisms of $(\Z_2^\infty)$.

\begin{thmx} \label{thm:B}
	The von Neumann algebra of the group $\Aut_{\textup{fin}}(\Z_2^\infty) \ltimes \Z_2^\infty$  admits a unique invariant subalgebra which does not come from a normal subgroup.
\end{thmx}
We would like to mention that another case of a group $G$ without the ISR property where all the invariant subalgebras of $L(G)$ can be determined, namely $G=SL_2(\Z) \ltimes \Z^2$, was treated recently in \cite{JiangLiu}.

As noted in \cite[Example~3.5]{amrutam2023invariant}, the lamplighter group $\Z_n\wr\Z$, with $n\ge 3$, does not possess the ISR property, a failure rooted in the presence of order $n$ elements in its normal abelian subgroup. This raises the natural question of whether a complete classification of invariant von Neumann subalgebras is possible when the abelian part consists only of order 2 elements. In this paper, we address this,  providing detailed information on invariant subalgebras of the group von Neumann algebra of the standard lamplighter group and showing that the latter does not enjoy the ISR property. Relying on classical ergodic-theoretic tools, including relatively weakly mixing factors of Bernoulli shifts, we establish the following theorem.
\begin{thmx} \label{thm:C}
Let $G=\mathbb{Z}_2\wr\mathbb{Z}=(\oplus_{\mathbb{Z}}\mathbb{Z}_2)\rtimes \mathbb{Z}$ be the classical lamplighter group. Let $A=\oplus_{\mathbb{Z}}\mathbb{Z}_2$.
Let $\mathcal{P}\subseteq L(G)$ be a $G$-invariant von Neumann subalgebra. Then the following hold true:
\begin{itemize}
\item[(i)] there exists a normal subgroup $N\lhd G$ such that $\mathcal{P}=(\mathcal{P}\cap L(A))\vee L(N)$, i.e.\ $\mathcal{P}$ is generated as a von Neumann algebra by $P\cap L(A)$ and $L(N)$; 
\item[(ii)] $\mathcal{P}\cap L(\mathbb{Z})=L(k\mathbb{Z})$ for some $k\in\mathbb{Z}$;
\item[(iii)] if the factor map $\pi: (\widehat{A},\text{Haar})\rightarrow (Y,\nu)$ is relatively weakly mixing, where $L^{\infty}(Y,\nu)\cong \mathcal{P}\cap L(A)$, then $\mathcal{P}=L^{\infty}(Y,\nu)\rtimes k\mathbb{Z}$;
\item[(iv)] $G$ does not have the ISR property.
\end{itemize}
	\end{thmx}
To understand the role of property (i) above one can mention that its obvious analogue holds also in the case of $SL_2(\Z) \ltimes \Z^2$	(\cite{JiangLiu}), but not in the case of $\Aut_{\textup{fin}}(\Z_2^\infty) \ltimes \Z_2^\infty$ considered in Theorem \ref{thm:B}.

We also investigate the \say{most obvious} semidirect product \( S_\infty \ltimes \mathbb{Z}_2^\infty \). This group also does not have the ISR property, which we prove, constructing certain explicit \say{exotic} invariant subalgebras of its group von Neumann algebra.

The general line of argument for Theorem \ref{thm:A}, Theorem \ref{thm:B} and the results on \( S_\infty \ltimes \mathbb{Z}_2^\infty \) is similar. We begin by noting that the semidirect products under consideration have only one non-trivial normal subgroup. We then proceed to analyze---essentially combinatorially---the possible invariant subalgebras of \( L(G) \), exploiting properties of conditional expectations, a natural variation of the argument characterising the factoriality via the ICC property and a version of the `character method' developed earlier in \cite{dudko2024character}.
As already mentioned, the techniques used to establish Theorem \ref{thm:C} are somewhat different, exploiting notions and methods of ergodic theory.

\subsection*{Organization of the paper}The detailed plan of the article is as follows: after this introduction, in Section \ref{sec:prelim} we discuss general properties of invariant subalgebras and corresponding conditional expectations. Section \ref{sec:ISR} treats the ISR  example 
$S(2^\infty) \ltimes \tilde{C}(X; \mathbb Z_2)$; here we prove Theorem \ref{thm:A}. In Section \ref{sec:uniqueexotic} we focus on $\Aut_{\textup{fin}}(\Z_2^\infty) \ltimes \Z_2^\infty$ and establish Theorem \ref{thm:B}. The classical lamplighter case is taken up in Section~\ref{Section: Lamplighter}. It should not come as a surprise, then, that Theorem~\ref{thm:C} is proven in this section. Finally, in Section \ref{sec:wreath} we gather certain observations concerning invariant subalgebras of $L(S_\infty\ltimes \Z_2^\infty)$ and present several examples.

\section{Preliminaries regarding invariant subalgebras} \label{sec:prelim}
In this short section we gather the conventions and basic facts concerning invariant subalgebras of group von Neumann algebras. We shall follow the following convention throughout the paper. 
\subsection*{Notation}
\begin{itemize}
    \item The scalar products will be linear on the right side.
    \item Throughout the paper, we will work with countable discrete groups.
   \item If $G$ is such a group, we will denote by $e$ its unit element.
   \item $L(G)$ is the von Neumann algebra associated to $G$, i.e.\ the von Neumann algebra generated in $B(\ell^2(G))$ by the shift operators $u_g, g \in G$, where $u_g \delta_{g'} = \delta_{gg'}$ for every $g,g' \in G$.
\end{itemize}
Recall that the canonical trace $\tau$ on $L(G)$ is defined via $\tau(x)=\langle \delta_e,x\delta_e\rangle$ for any $x\in L(G)$. Moreover, as $\tau$ is faithful, the map $L(G)\ni x\mapsto x\delta_e\in \ell^2(G)$ is injective. 
For any $x\in L(G)$, we may write $x=\sum_{g\in G}c_gu_g$ for its Fourier expansion (with $(c_g)_{g\in G}\in\ell^2(G)$), which simply means that $x\delta_e=\sum_{g\in G}c_g\delta_g$;
write $\text{supp}(x)=\{g\in G: c_g\neq 0\}$ for the \emph{support} of $x$. We will often simply write $g$ for $u_g$, viewing $g$ as a unitary in $L(G)$, if no confusion may arise.

A von Neumann subalgebra $\mathcal M\subseteq L(G)$ is said to be \emph{$G$-invariant} if for any $g \in G, m \in \mathcal{M}$ we have $u_g m u_g^* \in \mathcal{M}$; in other words $\mathcal{M}$ is invariant under the natural adjoint action of $G$. In such a case denote by $E: L(G)\rightarrow \mathcal M$ the unique $\tau$-preserving conditional expectation onto $\mathcal{M}$; the commutant of $\mathcal{M}$ is as usual denoted $\mathcal{M}'$. Let us recall some properties of $E$, needed to study invariant von Neumann subalgebras of $L(G)$. 
\begin{proposition}\label{prop: E properties}
	Let $\mathcal M\subseteq L(G)$ be a $G$-invariant von Neumann subalgebra, with the associated $\tau$-preserving conditional expectation $E: L(G)\rightarrow \mathcal M$. Then
	\begin{itemize}
		\item[(1)] $\tau(E(g)s)=\tau(E(g)E(s))=\tau(gE(s))$ for all $s, g\in G$;
		\item[(2)] $sE(g)s^{-1}=E(sgs^{-1})$ for all $s, g\in G$;
		\item[(3)] $sE(s^{-1})\in\mathcal M'\cap L(G)$ for any $s\in G$;
		\item[(4)] the formula $g \mapsto \tau(g^{-1}E(g))$ defines a character (normalised, $G$-invariant, positive-definite function) on $G$;
		\item[(5)] for any $g \in G$ we have $E(g)=0$ if and only if $\tau(g^{-1}E(g))=0$, and $E(g) = g$ if and only if $\tau(g^{-1}E(g))=1$.
	\end{itemize}
\end{proposition}
\begin{proof}
	All the above are well-known; for (4) see for example \cite[Proposition 3.2]{jz}.
\end{proof}

For any $G$ and a set of elements $S\subseteq G$ let $C_G(S):=\{ g \in G: gs = sg \textup{ for all } s \in S\}$ denote the \emph{centralizer} of $S$. Given an element $g\in G$ let us call an element $h\in G$ \emph{finitely permuted by the centralizer} of $g$ if $\sharp\{t^{-1}ht\mid t\in C_G(\{g\})\}<\infty$. Denote the set of such elements by $\fpc(g)$. The following statement seems well-known. For a proof, see \cite[Lemma~4.6]{dudko2024character}.
\begin{lemma}\label{lemma: fpc application} Let $\mathcal M\subseteq L(G)$ be a $G$-invariant von Neumann subalgebra, with the associated $\tau$-preserving conditional expectation $E: L(G)\rightarrow \mathcal M$. Then  $E(g)\in L(\fpc(g))$ for any $g\in G$.
\end{lemma}
If $N\subseteq G$ is a subgroup, it is easy to see that we can identify in a natural way $L(N)$ with a von Neumann subalgebra of $L(G)$ (and the $\tau$-preserving conditional expectation $E:L(G) \to L(N)$ is a natural extension of the map sending $u_g$ to itself for $g \in N$ and to $0$ for $g \in G \setminus N$). If $N$ is in addition \emph{normal}, $L(N)$ becomes an invariant von Neumann subalgebra of $L(G)$. The following definition is key for what follows.
\begin{definition}[ISR-property]
	We say that $G$ has the ISR (invariant subalgebra rigidity) property if every invariant von Neumann subalgebra $\mathcal M \subseteq L(G)$ is of the form $\mathcal M = L(N)$ for some normal subgroup $N \trianglelefteq G$.
\end{definition}

We will occasionally refer to an invariant von Neumann subalgebra $\mathcal M \subseteq L(G)$ as \emph{exotic} if it is not of the form $L(N)$ for $N\trianglelefteq G$. An infinite group $G$ with the ISR property must be necessarily ICC by \cite[Proposition 3.1]{amrutam2023invariant}.
Sometimes, we would like to understand when certain subspaces are preserved under the canonical conditional expectation. The following technical lemma allows us to do so.
\begin{proposition}\label{prop: E(S) subset S}
	Let $\mathcal P\subset\mathcal B$ be an inclusion of von Neumann algebras, let  $\mathcal\tau$ be a faithful trace on $\mathcal B$ and let $E:\mathcal B\to\mathcal P$ be the $\tau$-preserving conditional expectation. Let $\mathcal A\subset\mathcal B$ be a von Neumann subalgebra and let $S\subset\mathcal B$ be a subspace such that:
	\begin{itemize}
		\item $E(\mathcal A)\subset\mathcal A$;
		\item $E(S)\subseteq S+\mathcal A$;
		\item $\tau$ restricted to $S\mathcal A$ is equal to zero.
	\end{itemize}
	Then $E(S)\subseteq S$.
	\begin{proof}
		For $s\in S$, let us write $E(s)=t+A$ for some $t\in S$ and $A\in\mathcal{A}$. Applying $E$ on both sides, we get that
		\[t+A=E(s)=E(t)+E(A).\]
		Therefore, $E(t)=t+(A-E(A))$. Since $E(A)\in\mathcal{A}$, $A_2=A-E(A)\in\mathcal{A}$. Now, 
		\[E(t)E(t)^*=(t+A_2)(t^*+A_2^*)=tt^*+tA_2^*+A_2t^*+A_2A_2^*.\]
		Since $\tau|_{S\mathcal{A}}=0$, applying $\tau$ on both sides, we get that
		\[\tau(tt^*)+\tau(A_2A_2^*)=\tau\left(E(t)E(t)^*\right)\le \tau(E(tt^*))=\tau(tt^*).\]
		This is turn implies that $\tau(A_2A_2^*)=0$, hence, $A-E(A)=A_2=0$. Therefore, $E(t)=t$. Note that this implies $A=E(s-t)$. Now, since $\tau|_{S\mathcal{A}}=0$,
		\[\tau(AA^*)=\tau\left(E((s-t)A^*)\right)=\tau\left((s-t)A^*)\right)=0\]
		This shows that $A=0$, establishing the claim.
	\end{proof}
\end{proposition}\noindent
Note that the conditions of the above lemma imply that $S \cap \mathcal{A} =\{0\}$.

\section{Amenable ISR group with a non-trivial normal abelian subgroup} \label{sec:ISR}

In this section we construct an example of an amenable ISR group with a non-trivial normal abelian subgroup.

Consider the sequence of symmetric groups $(S(2^n))_{n \in \N}$ with inclusions $S(2^n)\to S(2^{n+1})$ given by: $s\to (s,s)$. Denote by $S(2^\infty)$ the inductive limit $\lim\limits_{n\to\infty}S(2^n)$. Observe that for each $n\in \N$ the group $S(2^n)$ acts as a group of permutations on $\{0,1\}^n$. We let $S(2^n)$ act on $X=\{0,1\}^\infty$ by: $s((x,y))=(s(x),y)$ for any $x\in \{0,1\}^n$, $y\in\{0,1\}^\infty$. Then $S(2^\infty)$ can be viewed as the group of homeomorphisms of $X$ whose action depends only on finitely many coordinates. For $g\in S(2^\infty)$ we will write $\suppg(g)=\{x\in X:g(x)\neq x\}$.

Let $C(X;\mathbb{Z}_2)=\{f:X\to \mathbb Z_2\;\text{continuous}\}$. Observe that each function in $C(X;\mathbb{Z}_2)$ can be written as 
\begin{equation}\label{equation: f_A}
	f_A(x)=\left\{\begin{array}{ll}0\in \mathbb Z_2,& x\in X\setminus A,\\ 1\in\mathbb Z_2,& x\in A, \end{array}\right.
\end{equation} where $A\subseteq X$ is a clopen set. We will view $C(X;\mathbb{Z}_2)$ as a group with multiplication given by pointwise addition (modulo 2); note that in terms of elements $f_A$, multiplication corresponds to taking the symmetric difference of sets. The group $S(2^\infty)$ acts naturally on $C(X)$. Let $S=\{f_X,f_\varnothing\}$. Then $S$ is a subgroup  of $C(X;\mathbb{Z}_2)$ fixed by the action of $S(2^\infty)$. Let $\widetilde C(X;\mathbb{Z}_2)=C(X;\mathbb{Z}_2)/S$. For a clopen set $A\subseteq X$ we will write  $\widetilde f_A=f_A+S$ for the relevant element of $\widetilde C(X;\mathbb{Z}_2)$. Observe that $\widetilde f_{X\setminus A}=\widetilde f_A$ for any clopen $A\subseteq X$; we will say that $A\subseteq X$ is \emph{non-trivial} if $A \neq \emptyset, A \neq X$. 
Throughout this section, we will set
\[ G=S(2^\infty)\ltimes \widetilde C(X;\mathbb{Z}_2).\]
We begin by analysing normal subgroups of $G$ and as it turns out, $G$ admits only one non-trivial normal subgroup.
\begin{proposition}
	The only non-trivial normal subgroup of $G$ is $\widetilde C(X;\mathbb{Z}_2)$.
\end{proposition}

\begin{proof}
		Let $N$ be a normal subgroup of $G$.

If $N\cap S(2^\infty)\neq \{e\}$ then $N\supseteq S(2^\infty)$, since $S(2^\infty)$ is simple.

If $N\cap \widetilde C(X;\mathbb{Z}_2)\neq\{e\}$ then $N\supseteq \widetilde C(X;\mathbb{Z}_2)$. Indeed, let $A\subseteq X$ be a non-trivial clopen set such that $\widetilde f_A\in N$. For any sufficiently large $n\in\mathbb N$ there exists a non-trivial subset $A_n\subseteq\{0,1\}^n$ such that $A=\bigcup\limits_{x\in A_n}[x],$ where 
\begin{equation}\label{equation: cylinder sets}[x]=\{(x,y):y\in \{0,1\}^\infty\}\subseteq X\end{equation} 
is the cylinder set of level $n$ corresponding to $x\in\{0,1\}^n$. In the group $S(2^n)$ of permutations on $\{0,1\}^n$ there exists a permutation $s$ such that $\#\{s(A_n)\Delta A_n\}=2$. Then, for any $y\in\{0,1\}^{n-1}$ there exists $t\in S(2^n)$ such that $t(s(A_n)\Delta A_n)=y\times \{0,1\}$. We obtain that $[y]=t(s(A)\Delta A)$ and so $\widetilde f_{[y]}\in N$. Since this is true for all sufficiently small cylinders in $X$, we obtain that $N\supseteq \widetilde C(X;\mathbb{Z}_2)$.
	
Now, assume  that $N\neq\{e\}$ and $N\neq\widetilde C(X;\mathbb{Z}_2)$. Then the last paragraph implies that $g\cdot\widetilde f_A\in N$ for some $g\in S(2^\infty)\setminus \{e\}$ and a (possibly trivial) clopen set $A\subseteq X$.  

	\vskip 0.1cm\noindent {\bf a)} Assume first that  $g(A)\neq A$. One has:
$$ N\ni g\widetilde f_A\cdot (g^{-1}\cdot g\widetilde f_A\cdot g)^{-1}=\widetilde f_{g(A)\Delta A}.$$ 
We obtain that $N\cap \widetilde C(X;\mathbb Z_2)\neq\{e\}$. Then, again by the last paragraph, $N\supseteq \widetilde C(X;\mathbb{Z}_2)$. Thus $\widetilde f_A\in N$ and further  $g\in N$. By the first lines of the proof, $N\supseteq S(2^\infty)$. It follows that $N=G$.

\vskip 0.1cm \noindent {\bf b)}  Assume now that $g(A)=A$. Either $A\cap \suppg(g)\neq\varnothing$ or $(X\setminus A)\cap \suppg(g)\neq\varnothing$. 
Take any $h\in S(2^\infty)$ supported inside either $A\cap \suppg(g)$ or $(X\setminus A)\cap \suppg(g)$ and such that $hg\neq gh$; it exists, as non-empty clopen sets in $X$ are necessarily big enough to support such $h$.
	Then 
	$$N\ni g\widetilde f_A\cdot (hg\widetilde f_A h^{-1})^{-1}=ghg^{-1}h^{-1}.$$ It follows that $N\cap S(2^\infty)\neq\{e\}$. By the first paragraph, $N\supseteq S(2^\infty)$. Take any non-trivial clopen set $B\subseteq X$. There exists $h\in S(2^\infty)$ such that $h(B)\neq B$. Then $\widetilde f_{h(B)\Delta B}=h\widetilde f_Bh^{-1}\widetilde f_B\in N$. By the second paragraph, $N\supseteq \widetilde C(X;\mathbb{Z}_2)$. Therefore, $N=G$.
	
\end{proof}

The next lemma will be used to `spread' subsets of $X$ using elements of $S(2^\infty)$.
\begin{lemma}\label{lemma: tilde f_A separation}
	For any clopen subsets $A,B\subseteq X$ with $B\notin\{\varnothing, X, A, X\setminus A\}$ there exists a sequence of elements $(h_n)_{n \in \N}$ of $S(2^{\infty})$ such that for each $n \in \N$ we have $h_n(A)=A$ and the sets $h_n(B)$ are pairwise distinct.
\end{lemma}
\begin{proof} 
	Replacing $A$ with $X\setminus A$ if necessary we may assume that $A\cap B\neq\varnothing$ and $A\setminus B\neq \varnothing$. Let $m\in\mathbb N$ be such that $A$ and $B$ are unions of cylinders of level $m$. For each $n\in\mathbb N$ pick a point $x_n\in\{0,1\}^{m+n}$ and a point $y_n\in\{0,1\}^{m+n}$ such that $$[x_n]\subseteq A\cap B,\;\;[y_n]\subseteq A\setminus B.$$
	Let $h_n\in S(2^{m+n})$ be the transposition of $x_n$ and $y_n$. It is not hard to see that $h_n(A)=A$ and the sets $h_n(B)$ are pairwise distinct for $n\in\mathbb N$.
\end{proof}
The next two results compute the sets $fpc(g)$ for particular elements $g \in G$.
\begin{lemma}\label{lemma: fpc widetilde f_A}
	Let $A\subseteq X$ be a non-trivial clopen set. Then $\fpc(\widetilde f_A)=\{e,\widetilde f_A\}$.
\end{lemma}
\begin{proof}
	Consider first an element $g\tilde{f}_B\in G$ with $g\in S(2^\infty),$ $g\neq e$, $B \subseteq X$ clopen. Let $x=(x_1,x_2,x_3,\ldots)\in X=\{0,1\}^\infty$ be such that $gx\neq x$. There exists $m\in \N$ such that $g\in S(2^m)$ and $g((x_1,\ldots,x_m))\neq ((x_1,\ldots,x_m))$. For $n\in\mathbb N$ let $B_n=[(x_1,\ldots,x_{n+m})]$ (a cylinder set of level $n+m$, see \eqref{equation: cylinder sets}). Then the elements $$\widetilde f_{B_n}\cdot g\widetilde f_B\cdot\widetilde f_{B_n}=\widetilde f_{B_n\Delta g(B_n)}\cdot g\widetilde f_B$$ are pairwise distinct. Since the elements $\widetilde f_{B_n}$ belong to the centralizer of $\widetilde f_A$ we obtain that $g\widetilde f_B\notin \fpc(\widetilde f_A)$. 
	
	Consider now an element $\widetilde f_B$ such that $B\notin\{\varnothing, X, A,X\setminus A\}$. Let $h_n$ be the sequence of elements from Lemma \ref{lemma: tilde f_A separation}. Then $h_n\in C_G(\widetilde f_A)$ and the elements $h_n\widetilde f_B h_n^{-1}=\widetilde f_{h_n(B)}$ are pairwise distinct. Thus, $\widetilde f_B\notin \fpc(\widetilde f_A)$.
\end{proof}
We now focus on computing the \fpc(s) for an involution $s$.
\begin{proposition}\label{prop: fpc of involution in S(2^infty)}
Let $s\in S(2^\infty)\subseteq G$ be an involution. Then $$\fpc(s)=\{e,s,\widetilde f_{\suppg(s)},s\widetilde f_{\suppg(s)}\}.$$
\end{proposition}
\begin{proof}
Consider $g \in S(2^\infty)$ and $A\subseteq X$ clopen such that $g\widetilde f_A\notin\{e,s,\widetilde f_{\suppg(s)},s\widetilde f_{\suppg(s)}\}$. We will construct a sequence of elements $h_n\in C_G(s)$ such that the elements $h_n\cdot g\widetilde f_A\cdot h_n^{-1}$ are pairwise distinct, which would imply that $g\widetilde f_A\notin \fpc(s)$. Take $m\in \N$ such that $g,s\in S(2^m)$ and $A$ is a union of cylinders of level $m$. Several cases are possible.\vskip 0.1cm\noindent
	${\bf a})$ $gs\neq sg$. Then, as $s$ is the product of disjoint transpositions, there exists $$x=(x_1,x_2,x_3,\ldots)\in \suppg(g)\cap \suppg(s)$$ such that $y=(y_1,y_2,y_3,\ldots):=sx\neq gx$. For $n\in\mathbb N$ let $h_n\in S(2^{n+m})$ be the transposition of $(x_1,\ldots,x_{n+m})$ and $(y_1,\ldots,y_{n+m})$. Then $h_ngh_n^{-1}\in S(2^{n+m})$ and $$h_ngh_n^{-1}((y_1,
	\ldots,y_{n+m}))=g((x_1,
	\ldots,x_{n+m})).$$ Let $k>n$ and $z\in [(y_1,\ldots,y_{n+m})]\setminus [(y_1,\ldots,y_{k+m})]$ (difference of cylinders of levels $n+m$ and $k+m$ around $y$, see \eqref{equation: cylinder sets}). Observe that by the choice of $x$ and $m$ the cylinders $[(x_1,\ldots,x_{n+m})]$, $[(y_1,\ldots,y_{n+m})]$, and $g([(x_1,\ldots,x_{n+m})])$ are pairwise disjoint. 
	Then 
	$h_n^{-1}(z)\in[(x_1,\ldots,x_{n+m})],$ 
	$h_k^{-1}(z)=z$ and so $$h_ngh_n^{-1}(z)\in g([(x_1,\ldots,x_{n+m})]),\;\; h_kgh_k^{-1}(z)\in g([(y_1,\ldots,y_{n+m})])\cup\suppg(h_k).$$ It follows that $h_ngh_n^{-1}(z)\neq h_kgh_k^{-1}(z)$. We obtain that $h_n g\widetilde f_Ah_n^{-1}\neq h_k g \widetilde f_Ah_k^{-1}$, and elements $h_n$ clearly commute with $s$. \vskip 0.1cm\noindent 
	${\bf b)}$ $gs=sg$ and  $\suppg(g)\setminus \suppg(s)\neq\varnothing$. Let $x\in\suppg(g)\setminus\suppg(s)$. Let $h_n\in S(2^{n+m+1})$ be the transposition of $(x_1,\ldots,x_{n+m},0)$ and $(x_1,\ldots,x_{n+m},1)$, which clearly commutes with $s$. Then for $k>n$ and $z\in [(x_1,\ldots,x_{n+m})]\setminus [(x_1,\ldots,x_{k+m})]$ one has  $h_kg(z) = h_n g(z) = g(z)$. Indeed, since $g\in S(2^m)$ and $x\in \suppg(g)$ we have that $g$ changes at least one of the first $m$ coordinates of $x$. Thus, $g[(x_1,\ldots,x_m)]\cap [(x_1,\ldots, x_m)]=\varnothing$. It follows that $gz\notin [(x_1,\ldots, x_m)]$ and does not belong to the support of either $h_n$ or $h_k$. This implies in turn that  $h_kgh_k^{-1}(z)=g(z)\neq gh_n(z)=h_ngh_n^{-1}(z)$, since $z$ and $h_n(z)$ have different $(n+m+1)$-st coordinate. Therefore, $h_n\widetilde gf_Ah_n^{-1}\neq h_k\widetilde  gf_Ah_k^{-1}$.
	\vskip 0.1cm\noindent
	${\bf c)}$ $gs=sg$, $\suppg(g)\subset\suppg(s)$, and $g\notin \{e,s\}$. The following two subcases are possible. 
	\vskip 0.1cm\noindent 
	${\bf c1)}$ $gx\neq sx$ for some $x\in\suppg(g)$. Let $u_n\in S(2^{n+m+1})$ denote the transposition of $(x_1,\ldots,x_{n+m},0)$ and $(x_1,\ldots,x_{n+m},1)$. Set $h_n=u_nsu_ns^{-1}$. Since $$s([(x_1,\ldots,x_m)])\cap [(x_1,\ldots,x_m)]=\varnothing$$ we have that $u_n$ and $su_ns^{-1}$ commute. Therefore, $h_n$ commutes with $s$. Arguing in the same way as in the case $b)$, we obtain that the elements $h_n g\widetilde f_Ah_n^{-1}$ are pairwise distinct for $n\in\mathbb N$.
	\vskip 0.1cm\noindent 
	${\bf c2)}$  $gx=sx$ for all $x\in\suppg(g)$. Then, as $g \neq e$, $g \neq s$, there  exist $x,y\in X$ such that $gx=sx\neq x$, $gy=y\neq sy$. For $n\in\mathbb N$ let $u_n\in S(2^{n+m})$ be the transposition of $(x_1,\ldots,x_{n+m})$ and $(y_1,\ldots,y_{n+m})$. Set $h_n=u_n\cdot (su_ns^{-1})$. Let $z\in [(x_1,\ldots,x_{n+m})]\setminus [(x_1,\ldots,x_{k+m})]$. Then $h_ngh_n^{-1}(z)=z\neq s(z)=h_kgh_k^{-1}(z)$. Once again, $h_n\widetilde gf_Ah_n^{-1}\neq h_k\widetilde  gf_Ah_k^{-1}$.\vskip 0.1cm\noindent
	${\bf d)}$ $g\in \{e,s\}$. It is not hard to see that $s\widetilde f_A\in \fpc(s)$ if and only if $\widetilde f_A\in \fpc(s)$. Therefore, without loss of generality in this case, we may assume that $g=e$. 
	\vskip 0.1cm\noindent 
	${\bf d1)}$ $s(A)\neq A$. Let $x\in A\setminus s(A)$. Let $h_n\in S(2^{n+m})$ be the transposition of $(x_1,\ldots,x_{n+m})$ and $s((x_1,\ldots,x_{n+m}))$, commmuting with $s$. Then $h_n(A)\setminus A=s[(x_1,\ldots,x_{n+m})]$.
	\vskip 0.1cm\noindent 
	${\bf d2)}$  $s(A)=A$, $A\cap\suppg(s)\notin\{\varnothing,\suppg(s)\}$. Let $x\in A\cap \suppg(s)$, $y\in \suppg(s)\setminus A$.  Let $u_n\in S(2^{n+m})$ be the transposition of $(x_1,\ldots,x_{n+m})$ and $(y_1,\ldots,y_{n+m})$. Set $h_n=u_n\cdot (s u_ns^{-1})$; one can check that $h_n s = sh_n$. Then $h_n(A)\setminus A=[(y_1,\ldots,y_{n+m})]\cup s[(y_1,\ldots,y_{n+m})]$. 
	\vskip 0.1cm\noindent
	${\bf d3)}$ $A\cap \suppg(s)\in \{\varnothing,\suppg(s)\}$. Our initial assumptions imply that $A\notin\{\varnothing,X,\suppg(s),X\setminus\suppg(s)\}$. We obtain that there exist points $x\in A\setminus\suppg(s),y\in X\setminus (A\cup \suppg(s))$. Let $h_n\in S(2^{n+m})$ be the transposition of $(x_1,\ldots,x_{n+m})$ and $(y_1,\ldots,y_{n+m})$. Then $h_n(A)\setminus A=[(y_1,\ldots,y_{n+m})]$.
	\vskip 0.1cm\noindent
	In the latter three cases we obtain that the elements $h_n\widetilde f_A h_n^{-1}=\widetilde f_{h_n(A)}$ are pairwise distinct for $n\in\mathbb N$. This finishes the proof. 
\end{proof}
In general, computing the set $fpc(s\widetilde f_A)$ is rather complicated, as the next two remarks show.
\begin{remark}\label{rem: fpc s widetilde f_A}
	Let $s\in S(2^\infty)$ be an involution and let $A\subseteq X$ be a clopen set such that $s(A)\cap A=\varnothing$. Let $g\in S(2^\infty)$ be such that $g(x)=s(x)$ for all $x\in A\cup s(A)$ and $g(x)=x$ othervise. Then $g\in\fpc(s\widetilde f_A)$. Thus, in general, it is not true that $\fpc(s\widetilde f_A)\subseteq s\cdot \widetilde C(X;\mathbb Z_2)\cup \widetilde C(X;\mathbb Z_2)$.
\end{remark}
\begin{remark}\label{rem: fpc s widetilde f_supp(s)}
	Let $s\in S(2^\infty)$ be an involution. Then $C_G(s)=C_G(s\widetilde f_{\suppg(s)})$. Therefore, $\fpc(s\widetilde f_{\suppg(s)})=\fpc(s)$.
\end{remark}
\subsection*{Invariant von Neumann subalgebras of $L(S(2^\infty)\ltimes \widetilde C(X;\mathbb{Z}_2))$}
In this subsection, we fix an invariant von Neumann subalgebra $\mathcal{M}\subseteq L(G)$, and denote by $E:L(G) \to \mathcal{M}$ the unique $\tau$-preserving conditional expectation. In many ways, $\widetilde C(X;\mathbb Z_2)$ governs the algebraic structure of $\mathcal{M}$ in the sense that either it has a trivial intersection with $\mathcal{M}$ or is completely contained inside $\mathcal{M}$.
\begin{proposition}\label{prop: M intersect L(widetilde C(X))}
We have either $\mathcal M\cap L(\widetilde C(X;\mathbb{Z}_2))=\mathbb C 1$ or $\mathcal M\supseteq L(\widetilde C(X;\mathbb{Z}_2))$.
\end{proposition}
\begin{proof}
Assume that $\mathcal{M} \cap L(\widetilde C(X;\mathbb{Z}_2)) \ne \mathbb{C}1$. Then there exists a non-scalar element $m \in \mathcal{M} \cap L(\widetilde C(X;\mathbb{Z}_2))$. Expand $m$ in its Fourier series as
\begin{equation}\label{eq:fourier_expansion}
m = \sum_{\widetilde f \in \widetilde C(X;\mathbb{Z}_2)} m_{\widetilde f} \cdot \delta_{\widetilde f}.
\end{equation} Subtracting $m_e\cdot 1$ from $m$ we may assume that $m_e=0$. 
Since $m \notin \mathbb{C}1$, there exists a nontrivial clopen set $A \subsetneq X$ such that $m_{\widetilde f_A} \ne 0$. For convenience, normalize $m$ so that $m_{\widetilde f_A} = 1$. Assume that $m\neq \delta_{\widetilde f_A}$. Choose $\widetilde f_B \in \text{supp}(m)$ with $\widetilde f_B \ne \widetilde f_A$ such that
\[
|m_{\widetilde f_B}| = \max\{|m_{\widetilde f}| : \widetilde f \ne \widetilde f_A\}.
\]
Using  Lemma \ref{lemma: tilde f_A separation} choose a sequence $(h_n)_{n \in \mathbb{N}} \subset S_{2^\infty}$ such that
\begin{itemize}
    \item $h_n^{-1}(A) = A$ for all $n$,
    \item the sets $h_n^{-1}(B)$ are pairwise distinct (modulo complements).
\end{itemize}
The Fourier coefficients of $ u_{h_n} m u_{h_n}^{-1} \in \mathcal{M}$ satisfy
\[ (u_{h_n} m u_{h_n}^{-1})_{\widetilde f} = m_{h_n^{-1}\widetilde f h_n}.\]
Passing to a subsequence, we may assume that there exists a limit in the weak operator topology $m' = \mathrm{w\text{-}lim}_{n \to \infty}u_{h_n} m u_{h_n}^{-1}\in\mathcal M$. Clearly, $m'_{\widetilde f_A}=m_{\widetilde f_A}=1$. Since $\{m_{\widetilde f}:\widetilde f\in \widetilde C(X;\mathbb{Z}_2)\}$ is square-summable, $$m_{\widetilde f_B}=\lim\limits_{n\to \infty}m_{\widetilde f_{h_n^{-1}(B)}}=0.$$

Fix for the moment the sequence $(h_n)_{n=1}^\infty$. Next, let us order all elements $\widetilde f\in \widetilde C(X;\mathbb{Z}_2)$ in an arbitrary way: $$\widetilde C(X;\mathbb{Z}_2)=\{\widetilde f_1,\widetilde f_2,\widetilde f_3,\ldots\}.$$ Set $S_0 = \mathbb{N}$. By induction, we will construct an increasing sequence $(n_k)_{k \in \N}$ of positive integers and a decreasing sequence of infinite subsets $S_k\subseteq\mathbb N$ such that the following holds for each $k\in \N$ 
	\begin{itemize}
		\item[$1)$] $n_k\in S_{k-1}$;
		\item[$2)$] $\min S_k>n_k$;
		\item[$3)$] either $h_j^{-1}\widetilde f_kh_j$ does not depend on $j$ for $j\in S_k$ or $h_j^{-1}\widetilde f_k h_j$ are pairwise distinct for $j\in S_k$.
	\end{itemize} 
The construction of the starting pair, $n_1$ and $S_1$, can proceed as below, so we will not repeat it. Assume that for a fixed $k \in \N$ we have constructed $n_1,\ldots,n_k$ and $S_1,\ldots, S_k$ satisfying the inductive hypothesis. If the number of distinct elements in $\{h_j^{-1} \widetilde f_{k+1} h_j:j\in S_k\}$ is infinite there exists an infinite subset $S'_{k+1}\subseteq S_k$ such that $h_j^{-1}\widetilde f_{k+1} h_j$ are pairwise distinct for $j\in S_{k+1}'$. Otherwise, there exists an infinite subset $S_{k+1}'\subseteq S_k$ such that $h_j^{-1}\widetilde f_{k+1} h_j$ are all equal for $j\in S_{k+1}'$. In both cases, set $$n_{k+1}=\min S_{k+1}',\;\;S_{k+1}=S'_{k+1}\setminus \{n_{k+1}\}.$$  The inductive hypothesis is satisfied for $k+1$ by construction.

Now, replace $h_n$ with the  subsequence $(h_{n_k})_{k \in \N}$. Then for any element $\widetilde f\in\widetilde C(X;\mathbb{Z}_2)$ the following is true: either  $h_n^{-1}\widetilde f h_n$ becomes constant for large $n\in \N$ or  the elements $\{h_n^{-1}\widetilde f h_n:n\geqslant N\}$ are pairwise distinct for $n\geqslant N$ for sufficiently large $N\in \N$. Recall that $m'$ is the weak limit of $u_{h_n}mu_{h_n}^*$. In particular, $m'_{\widetilde f}=\lim_{n \to \infty} m_{h_n^{-1}\widetilde fh_n}$ for any $\widetilde f$. We obtain for each $\widetilde f$ that either  $m'_{\widetilde f}=m_{h_n\widetilde f h_n^{-1}}$ for large enough $n$ or $m'_{\widetilde f}=0$. 
	
By performing the operation $m\to m'$ repeatedly starting with $m$ we obtain a sequence of elements $m^{(n)}\in \mathcal M\cap L(F_2^\infty)$ such that $m^{(n)}_{\widetilde f_A}=m_{\widetilde f_A}=1$ for all $n$ and for any $\widetilde f\neq \widetilde f_A$ one has $m^{(n)}_{\widetilde f}=0$ for all sufficiently large $n$. It follows that $m^{(n)}$ converges weakly to $\widetilde f_A$. Therefore, $\widetilde f_A\in \mathcal M$.

Since $\delta_{\widetilde f_A} \in \mathcal{M}$ and $\mathcal{M}$ is invariant under conjugation by $G = S(2^\infty) \ltimes \widetilde C(X;\mathbb{Z}_2)$, we get:
\begin{equation}\label{eq:closure_under_conjugation}
\delta_{g(\widetilde f_A)} = u_g \delta_{\widetilde f_A} u_g^{-1} \in \mathcal{M}, \quad \forall g \in S(2^\infty).
\end{equation}
The set $\{\delta_{\widetilde f_C} : C \subseteq X \text{ clopen, nontrivial}\}$ spans a dense *-subalgebra of $L(\widetilde C(X;\mathbb Z_2))$. Hence,
\begin{equation}\label{eq:algebra_generated}
L(\widetilde C(X;\mathbb Z_2)) \subseteq \mathcal{M}.
\end{equation}
This concludes the proof.
\end{proof}
Moreover, it follows immediately from Lemma \ref{lemma: fpc application} and Lemma \ref{lemma: fpc widetilde f_A} that under the canonical conditional expectation $E$, $L(\widetilde C(X;\mathbb Z_2))$ remains preserved. The following proposition makes it precise.
\begin{proposition}\label{cor: E(widetilde C(X))}
	One has $E(L(\widetilde C(X;\mathbb Z_2)))\subseteq L(\widetilde C(X;\mathbb Z_2))$.
\end{proposition}
We now examine the image of an involution under the canonical conditional expectation $E$.
\begin{proposition}
	Let $s\in S(2^\infty)\subseteq G$ be an involution. Then $$E(s)\in\{s,\tfrac{1}{2}s(\id+\widetilde f_{\suppg(s)}),\tfrac{1}{2}s(\id-\widetilde f_{\suppg(s)})\}\cup \mathbb C s\widetilde f_{\suppg(s)}.$$
\end{proposition}
\begin{proof}
	We may and do assume that $s \neq \id$.
	Let $\mathcal A=L(\{\id,\widetilde f_{\suppg(s)}\})$ and $\mathcal S=s\mathcal A$. By Lemma \ref{lemma: fpc widetilde f_A}, one has $E(\mathcal A)\subseteq\mathcal A$. By Proposition \ref{prop: fpc of involution in S(2^infty)}, Remark \ref{rem: fpc s widetilde f_supp(s)}, and Lemma \ref{lemma: fpc application}, one has $E(\mathcal S)\subseteq\mathcal S+\mathcal A$.
	Clearly, $\tau$ is equal to zero on $\mathcal S\mathcal A=\mathcal S$. By Proposition \ref{prop: E(S) subset S}, $E(\mathcal S)\subseteq\mathcal S$. It follows that 
	$ E(s)=\lambda_1 s+\lambda_2 s\widetilde f_{\suppg(s)}.$ We will now consider several cases.
	\vskip 0.1cm\noindent
	${\bf a)}$ $\lambda_2=0$. Then
	either $\lambda_1=0$ and $E(s)=0$ or $\lambda_1\neq 0$, $s\in \mathcal M$ and $E(s)=s$.
	\vskip 0.1cm\noindent
	${\bf b)}$ $\lambda_1=0$. Then $E(s)=\lambda_2 s\widetilde f_{\suppg(s)}$.
	\vskip 0.1cm\noindent   
	${\bf c)}$ Assume that $\lambda_1\lambda_2\neq 0$. Then one has:
	$$E(s)E(s)^*=E(s)^2=(\lambda_1\id +\lambda_2 \widetilde f_{\suppg(s)})^2\in L(\widetilde C(X;\mathbb Z_2))\setminus \mathbb C\id.$$ Using Proposition \ref{prop: M intersect L(widetilde C(X))}, we obtain that $\mathcal M\supseteq L(\widetilde C(X;\mathbb Z_2))$. Therefore,
	$$s\cdot(\lambda_1 +\lambda_2\widetilde f_{\suppg(s)})=E(s)=E(E(s))=E(s)(\lambda_1 +\lambda_2\widetilde f_{\suppg(s)})=s\cdot (\lambda_1 +\lambda_2\widetilde f_{\suppg(s)})^2.$$ It follows that $P:=\lambda_1\id +\lambda_2\widetilde f_{\suppg(s)}$ is a projection, which is easily seen to be selfadjoint, from which it is straightforward to obtain that 
	$$P\in\{0,\id,\tfrac{1}{2}(\id+\widetilde f_{\suppg(s)}),\tfrac{1}{2}(\id-\widetilde f_{\suppg(s)})\}.$$ 
	This finishes the proof.
\end{proof}

We are ready to formulate and prove the main result of this section.
\begin{theorem}\label{th: S(2^infty) and ISR}
	The group $G=S(2^\infty)\ltimes \widetilde C(X;\mathbb{Z}_2)$ is an amenable ICC group with the ISR property containing a non-trivial abelian normal subgroup. 
\end{theorem} 
\begin{proof}
The only non-trivial statement is the one concerning the ISR property (as the ICC property follows then by \cite[Proposition 3.1]{amrutam2023invariant}). Let $\mathcal{M}\subseteq L(G)$ be an invariant von Neumann subalgebra  and denote by $E:L(G) \to \mathcal{M}$ the unique $\tau$-preserving conditional expectation.

We will consider several cases.

\noindent ${\bf 1)}$ Assume first that $E(s)=s$ for some $s\in S(2^\infty)\setminus e$. Then, by simplicity of $S(2^\infty)$, one has $E(g)=g$ for all $g\in S(2^\infty)$. Therefore, for every clopen subset $A\subseteq X$ one has $E(g\widetilde f_A)=gE(\widetilde f_A)$. Take a non-trivial clopen subset $A\subseteq X$, and any $g\in S(2^\infty)$ such that $gA\neq A$. We have:
$$g\widetilde f_{A\Delta g^{-1}(A)}=\widetilde f_A g\widetilde f_A=\widetilde f_A E(g)\widetilde f_A=E(\widetilde f_A g\widetilde f_A)=E(g\widetilde f_{A\Delta g^{-1}(A)})=gE(\widetilde f_{A\Delta g^{-1}(A)}).$$ 
It follows that $E(\widetilde f_{A\Delta g^{-1}(A)})=\widetilde f_{A\Delta g^{-1}(A)}$. By Proposition \ref{prop: M intersect L(widetilde C(X))}, $\mathcal M\supseteq L(\widetilde C(X;\mathbb Z_2))$. Thus, in this case, $\mathcal M=L(G)$.

\vskip 0.1cm\noindent
	${\bf 2)}$
Now, assume that $E(g)\neq g$ for all $g\in S(2^\infty)\setminus\{e\}$ and that $E(s)\in\mathbb Cs\widetilde f_{\suppg(s)}$ for some involution $s\in S(2^\infty),s\neq e$, with $\suppg(s)\neq X$. Consider the character $\chi(g)=\tau(g^{-1}E(g))$, $g \in G$. Its restriction to $S(2^\infty)$ remains a character, 
and indecomposable characters on $S(2^\infty)$ were described in \cite{dudko_jfa2011}. They are indexed by the parameter set $k\in\mathbb{Z}_+^\infty:=\mathbb Z_+\cup\{\infty\}$. For $k\in \mathbb Z_+$ we have
$$\chi_k(g)=\mu(\fix(g))^k, \;\;\; g \in S(2^\infty),$$ 
where $\fix(g)=X\setminus \suppg(g)$ is the set of fixed points of $g$, and $\mu=\{1/2,1/2\}^\infty$ is the uniform Bernoulli measure on $X$. On the other hand $\chi_\infty$ is the regular character on $S(2^\infty)$. It follows from the standard decomposition theory for characters (see \cite{Takesaki-OperatorAlgebras}, Chapter IV, Theorem 8.21) that there exists a family of non-negative numbers $(\alpha_k)_{k\in\mathbb{Z}_+^\infty }$ summing to $1$ such that $$\chi=\sum\limits_{k\in\mathbb{Z}_+^\infty}\alpha_k\chi_k.$$ 
Since $\chi(s)=\lambda\tau(\widetilde f_{\suppg(s)})=0$ and $\mu(\fix(s))\neq 0$ we obtain that $\alpha_k=0$ for all $k\in\mathbb Z_+$ and thus $$\tau(g^{-1}E(g))=\chi(g)=\chi_\infty(g)=\delta_{g,e}$$ for $g\in S(2^\infty)$. Standard arguments show that $E(g)=0$ for all $g\in S(2^\infty)\setminus\{e\}$. 
	
Next, let $g\in S(2^\infty)\setminus \{e\}$ and let $A\subseteq X$ be a clopen subset.
Let $m\in \N$ be such that $g\in S(2^m)$ and $A$ is a union of cylinders of level $m$. Take a point $x\in\suppg(g)$. Denote for $n \in \N$  by $h_n\in S(2^{m+n})$ the transposition of $(x_1,x_2,\ldots,x_{m+n},0)$ and $(x_1,x_2,\ldots,x_{m+n},1)$.
Then $h_n\in S(2^\infty)$ are such that $g_n:=h_n\cdot g\widetilde f_A\cdot h_n^{-1}$ are pairwise distinct and $h_n(A)=A$. In particular, $g_ng_k^{-1}\in S(2^\infty)\setminus \{e\}$ and so $\chi(g_ng_k^{-1})=0$ for all $k, n \in \N$, $n\neq k$. Applying \cite[Lemma 2.7]{dudko2024character} we obtain that $\chi(g\widetilde f_A)=0$. By Proposition \ref{prop: E properties} (5) it follows that $E(g\widetilde f_A)=0$. Thus $E(h)=0$ for all $h\notin \widetilde C(X;\mathbb{Z}_2).$ Taking into account Corollary \ref{cor: E(widetilde C(X))} and Lemma \ref{prop: M intersect L(widetilde C(X))}, we obtain that either $\mathcal M=\mathbb C1$ or $\mathcal M=L(\widetilde C(X;\mathbb Z_2))$.

\vskip 0.1cm\noindent
${\bf 3)}$ Assume that neither $1)$ nor $2)$ holds. By Proposition \ref{prop: fpc of involution in S(2^infty)}, for all involutions $s\in S(2^\infty)$ with $\suppg(s)\notin\{\varnothing,X\}$ one has $E(s)=\tfrac{1}{2}s(\id\pm \widetilde f_{\suppg(s)})$. For $n\in\mathbb N$ and $x,y\in \{0,1\}^n$ denote by $(x,y)\in S(2^n)$ the transposition of $x$ and $y$. Let $s=(000,100)\in S(2^3)\subseteq S(2^\infty)$ and $g=(11,10)\subseteq S(2^2)\subseteq S(2^\infty)$. Recall that for $x\in \{0,1\}^n,n\in\mathbb Z_+,$ $[x]$ stands for the corresponding cylinder subset of $X$ (see \eqref{equation: cylinder sets}). One has 
$$E(s)=\tfrac{1}{2}s(\id\pm \widetilde f_{[000]\cup[100]}),\;\;E(g)=\tfrac{1}{2}g(\id\pm \widetilde f_{[11]\cup[10]}).$$ Since $s^{-1}E(s)\in\mathcal M'$, we obtain that $\widetilde f_{[000]\cup[100]}$ commutes with $E(g)$. We have:
\begin{align*} E(g)\widetilde f_{[000]\cup[100]}=\tfrac{1}{2}g\cdot(\id\pm \widetilde f_{[11]\cup[10]})\cdot \widetilde f_{[000]\cup[100]}=\tfrac{1}{2}g( \widetilde f_{[000]\cup[100]}\pm \widetilde f_{[11]\cup[101]\cup [000]}),\\
\widetilde f_{[000]\cup[100]} E(g)=\tfrac{1}{2}g\cdot \widetilde f_{[000]\cup[110]}\cdot(\id\pm \widetilde f_{[11]\cup[10]})=\tfrac{1}{2}g( \widetilde f_{[000]\cup[110]}\pm \widetilde f_{[111]\cup[10]\cup [000]}).\end{align*} We see that the two rows in the above formula are not equal elements of $L(G)$. This contradiction shows that case ${\bf 3)}$ is impossible.
\end{proof}
\section{Finitary automorphisms of $\mathbb{Z}_2^\infty$ and a group von Neumann algebra with a unique `exotic' invariant subalgebra}
\label{sec:uniqueexotic}

This section is devoted to the study of the semidirect product for the action on $\Z_2^\infty$ given by all \say{finitary} automorphisms.

Let $F_2^\infty:=\bigoplus_{\mathbb{N}} F_2$ be the vector space of infinite columns of elements of the field $F_2$ having finitely many nonzero entries. This vector space also can be viewed as a group with multiplication given by addition of vectors. By definition, $GL(\infty, F_2)$ is the group of all infinite invertible matrices $M$ over $ F_2$  such that $\#\{(i,j)\in \N^2:M_{ij}\neq \delta_{i,j}\}<\infty$. This group acts by matrix multiplication from the left on $F_2^\infty$. 
For $v\in F_2^\infty$ and $g\in GL(\infty, F_2)$ we denote by $g(v)$ the image of $v$ under the action of $g$. Consider the semidirect product $GL(\infty,F_2)\ltimes F_2^\infty$. We will write elements of this group as $gv$ or $g\cdot v$, $g\in GL(\infty,F_2),v\in F_2^\infty$, where $\cdot$ is used to visually separate the terms of a product of elements of a group. We have: \begin{equation}\label{eq: product in H}
	g_1v_1\cdot g_2v_2=g_1g_2\cdot (g_2^{-1}(v_1)+v_2),\;\;g_1,g_2\in GL(\infty,F_2),\;v_1,v_2\in F_2^\infty.
\end{equation}
 One of the motivations to consider the group $GL(\infty,F_2)$ is that it can be viewed as a certain group of automorphisms of $\mathbb Z_2^\infty$. Let $\Aut(\mathbb Z_2^n)$ be the group of all automorphisms of $\mathbb Z_2^n$. Since $S(n)$ acts naturally on $\mathbb Z_2^n$ by permuting coordinates we have $\mathrm{Aut}(\mathbb Z_2^n)\supseteq S(n)$. Notice that for $n\geqslant 2$ the inclusion is strict, as the next example shows.
\begin{example} The map $\varphi:\mathbb Z_2^2\to\mathbb Z_2^2$ given by $\varphi((0,0))=(0,0),\varphi((1,0))=(1,1),\varphi((0,1))=(0,1),\varphi((1,1))=(1,0)$ is an automorphism of $\mathbb Z_2^2$, which does not arise from any element of $S_2$ in the manner described above.
\end{example}

Considering natural inclusions $\mathbb Z_2^n=\mathbb Z_2^n\times \{0\}\subseteq \mathbb Z_2^{n+1}$, $n \in \N$, we obtain inclusions $\Aut(\mathbb Z_2^n)\subseteq\aut(\mathbb Z_2^{n+1})\subseteq\Aut (\Z_2^\infty)$. Let $\Aut_\fin (\Z_2^\infty)=\bigcup_{n\in\mathbb N}\aut(\mathbb Z_2^n)<\aut (\Z_2^\infty)$. Then $\Aut_\fin (\Z_2^\infty)\supseteq S_\infty$.

If $\mathcal I:F_2^\infty\to \mathbb Z_2^\infty$ denotes the tautological map forgetting the multiplication in $F_2$, we see that the map $g\to\mathcal I(g):=\mathcal I\circ g\circ \mathcal I^{-1}$ provides an isomorphism from $GL(\infty, F_2)$ onto $\Aut_\fin (\mathbb Z_2^\infty)$. Therefore, semidirect  products $\Aut_\fin (Z_2^\infty)\ltimes Z_2^\infty$  and $GL(\infty,F_2)\ltimes F_2^\infty$ are isomorphic. For the rest of this section set 
\[ G = \Aut_\fin (\Z_2^\infty)\ltimes \Z_2^\infty \cong GL(\infty,F_2)\ltimes F_2^\infty.\]

	For $g\in GL(\infty, F_2)<G$ and $v\in F_2^\infty<G$ we have: $g\cdot v\cdot g^{-1}=g(v)$ (see \eqref{eq: product in H}).
This characterization of $G$ allows us to conclude that there is only one normal subgroup of $G$.
\begin{lemma}\label{lemma: normal subgroups of H}
	The group $F_2^\infty$ is the only non-trivial normal subgroup of $G$.
\end{lemma}
\begin{proof}
Let $N$ be a normal subgroup of $G$. 

If $v\in N$ for some $v\in F_2^\infty\setminus \{0^\infty\}$, then taking into account that $GL(\infty,F_2)$ acts transitively on $F_2^\infty\setminus \{0^\infty\}$ we obtain that $F_2^\infty<N$. Then either $N=F_2^\infty$ or $N$ contains an element in $G\setminus F_2^\infty$.

Assume now that $gv\in N$ for some $g\in GL(\infty,F_2)\setminus \{e\}$, $v\in F_2^\infty$. Take any $w\in F_2^\infty$ such that $gw\neq w$. Then $(-w)\cdot gv\cdot w\cdot (gv)^{-1}=(g(w)-w)\in N$. By the previous paragraph, we obtain that $F_2^\infty<N$. Moreover, $g=gv\cdot(-v)\in N$. Using simplicity of $GL(\infty,F_2)$, which is a consequence of the results in \cite[Chapter 8]{Rotman_book} we deduce that $GL(\infty,F_2)<N$. Thus, in this case $N=G$, which finishes the proof.
\end{proof}	
	 
The following statement is straightforward.
\begin{lemma}\label{lemma: disjoint subsets}
	For any distinct non-zero vectors $v_1,v_2\in F_2^\infty$ and any distinct non-zero vectors $w_1,w_2\in F_2^\infty$ there exists $g\in GL(\infty, F_2)$ such that $g(v_1)=w_1$, $g(v_2)=w_2$.
\end{lemma}

The next lemma concerns the sets $\fpc(v)$ for $v\in F_2^\infty$.

\begin{lemma}\label{lem:fpc(v)}
	For any nonzero vector $v\in F_2^\infty$ one has $\fpc(v)=\{v,e\}$.
\end{lemma}
	\begin{proof} 
	Let $g\in GL(\infty,F_2)$, $w \in F_2^\infty$. Suppose that $g\cdot w\notin\{e,v\}$. Consider first the case 
	$v\neq w$. Lemma \ref{lemma: disjoint subsets} implies that there exists a sequence of elements $h_n\in GL(\infty,F_2)$ such that $h_n(v)=v$, but $h_n(w)$ are pairwise distinct for $n\in\mathbb N$. From this we conclude that $h_n\in C_H(\{v\})$, but  the elements $h_n\cdot gw\cdot h_n^{-1} = h_ngh_n^{-1}\cdot h_n(w)$  are pairwise distinct in $H$. This implies that $g \cdot w\notin\fpc(v)$.

	Let then $g\neq e$, $w=v$ and assume that $R(g-I)\neq\{v,0\}$. There exists $u\in F_2^\infty$ such that $g(u)-u\notin \{v,e\}$. Lemma \ref{lemma: disjoint subsets} implies that there exists a sequence of elements $h_n\in GL(\infty,F_2)$ such that $h_n(v)=v$ (and thus $h_n\in C_H(\{v\}))$, but $h_n(g(u)-u)$ are pairwise distinct for $n\in\mathbb N$. Observe that $h_ngh_n^{-1}(h_n(u))=h_n(g(u))$, and thus $h_n(g(u)-u)$ belongs to the range of $h_ngh_n^{-1}-I$. Since for every element $t\in GL(\infty,F_2)$ the matrix $t-I$ has finite range we conclude that  among the elements $h_n\cdot g\cdot v\cdot h_n^{-1}=h_n gh_n^{-1}\cdot h_n(v) = h_n gh_n^{-1}\cdot v $ there is an infinite subset consisting of pairwise distinct elements of $H$. Therefore, $g\cdot v\notin\fpc(v)$.
	
	Finally, assume that $w=v$ and $R(g-I)=\{v,0\}$. Then $g-I=v\cdot r$ for some finitely supported row-vector $r$. Let $m\in \N$ be such that $r_n=v_n=0$ for all $n>m$. Take $k\in \N$ such that $r_k\neq 0$. Set $h_n=I+e_{k,n}$ for $n>m$. Here, $e_{k,n}$ denotes the matrix whose only non-zero entry is at the $(k,n)$-th position. Then $$h_n^{-1}=h_n,\;h_nv=v,\;h_n(g-I)=g-I,\;h_ngh_n^{-1}-I=(g-I)h_n^{-1}=v\cdot r\cdot h_n=v\cdot (r+\delta_n).$$ In particular, $h_ngh_n^{-1}-I$ has $n$-th column equal to $v$, and at the same time all $l$-th columns for $l>m, l \neq n$ equal $0$. It follows that $h_ngh_n^{-1}$ are pairwise distinct for $n>m$ and so $g\cdot v\notin\fpc(v)$. 
\end{proof}

The following factorization lemma will be used later.

	\begin{lemma}\label{lemma: factorization in GL}
	Let $g\in GL(\infty,F_2)\setminus\{e\}$. Then there exists $k\in\mathbb N$ and elements $s_1,\ldots,s_k\in GL(\infty,F_2)$ such that 
	\begin{itemize}
		\item $s_i$ is conjugate to $s=(1,2)$ in $GL(\infty,F_2)$, $i=1,\ldots,k$,
		\item $s_1s_2\cdots s_k=g$,
		\item $R(s_1-I)+R(s_2-I)+\ldots+R(s_k-I)=R(g-I)$.
	\end{itemize}
\end{lemma}
\begin{proof}
Set $n=\dim R(g-I)$. Let $h\in GL(\infty,F_2)$ be such that $h(R(g-I))=F_2^n$. Replacing $g$ by $hgh^{-1}$ without loss of generality we will assume that $R(g-I)=F_2^n$. Then $g$ is of the form 
$\begin{bmatrix}
			g_1 & m \\0 &I
\end{bmatrix}$, where $g_1 \in GL(n,F_2)$ and $m \in F_2^{n\times\infty}$. Assume first that $m \neq 0^\infty$. Observe that $$\begin{bmatrix}
			0&1\\1&1
		\end{bmatrix}
		\begin{bmatrix}
			0&1\\1&0
		\end{bmatrix}
		\begin{bmatrix}
			0&1\\1&1
		\end{bmatrix}^{-1}=
		\begin{bmatrix}
			1&1\\0&1
		\end{bmatrix}.$$
The latter matrix is conjugate to every matrix of the form $I+E_{ij}$, $i<j$. Clearly, there exist $l\in \mathbb{N}$ and matrices $s_1,\ldots,s_l$ of the form $I+E_{ij}$, $1\leqslant i<j$, $i \leqslant n$, such that $s_1s_2\cdots s_l=
		\begin{bmatrix}
			I_n&m\\0&I
		\end{bmatrix}$. 
		
Since $GL(n,F_2)=PSL(n, F_2)$ for $n\geq 3$ is simple \cite[Chapter 8]{Rotman_book}, $g_1$ can be written as a product $s_{l+1}\cdots s_k$ of elements from $GL(n,F_2)$ conjugate to $s$, for some $k\geqslant l$. We obtain:
		$$g=\begin{bmatrix}
			g_1 & m \\0 &I
		\end{bmatrix}=
		\begin{bmatrix}
			I_n&m\\0&I
		\end{bmatrix}
		\begin{bmatrix}
			g_1 & 0\\0 & I
		\end{bmatrix}=s_1s_2\cdots s_k.$$
If $n=1$, then $g_1=1$ and $g=\begin{bmatrix}
			I_n&m\\0&I
		\end{bmatrix}$ 
with $m \neq 0^\infty$, so we can simply use the argument above. Finally if $n=2$, then $g_1\in GL(2, F_2)$ and it is not hard to check that if $g_1 \neq e$ then one can always write  
$\begin{bmatrix}
			g_1 & 0\\0 & I
		\end{bmatrix}$ 
as a product of at most two elements which are both conjugate to $s$. Hence the above decomposition of $g$ as a product of two matrices shows $g$ is indeed a product of several $s_i$, elements which are all conjugate to $s$ for all $n\geq 1$.

By construction, $R(s_i-I)\subseteq F_2^n$ for each $1\leqslant i\leqslant k$. Since for every $g,h\in GL(\infty,F_2)$ one has $R(g_1g_2-I)\subseteq R(g_1-I)+R(g_2-I)$, we have:
$$F_2^n\subseteq R(g-I)\subseteq R(s_1-I)+R(s_2-I)+\ldots+R(s_k-I)\subseteq F_2^n,$$ 
which implies that the latter inclusion is, in fact, an identity. This finishes the proof.
	\end{proof}

Finally we shall quote the description of indecomposable characters of $G$, contained in  \cite{Dudko-2008}, Theorem 1, as a particular case. More preciesly, the above mentioned theorem gives description of indecomposable characters on certain class of groups of infinite matrices $M_{n,m}$ over a finite field $F_q$. The group $G$ is isomorphic to $M_{n,m}$ in the case $n=1$, $m=0$, $q=2$. 

\begin{theorem}\label{theorem: characters Gl times F2infty}
Indecomposable characters on $GL(\infty,F_2)\ltimes  F_2^\infty$ are indexed by $\Z_+^\infty \times \{0,1\}$. For every $(k,d) \in \Z_+^\infty \times \{0,1\}$ and $g\in GL(\infty,F_2),v\in  F_2^\infty$ we have  
\begin{equation}\chi_{k,1}(g\cdot v)=2^{-k\,\mathrm{rank}(g-I)},\end{equation}
	\begin{equation}
		\chi_{k,0}(g\cdot v)=\left\{\begin{array}{ll}2^{-k\,\mathrm{rank}(g-I)},&\text{if}\;\;v\in R(g-I),\\
			0,&\text{otherwise}
		\end{array}.\right.
	\end{equation}
\end{theorem}

\subsection*{Invariant von Neumann subalgebras of $L(\aut_\fin (\Z_2^\infty)\ltimes \Z_2^\infty)$}

We will now study invariant von Neumann subalgebras in $L(G)$. First, however, consider just the `acting' group, $GL(\infty,F_2)$. We show that it satisfies the ISR-property.
\begin{proposition}\label{prop:(GLinfty,F2)_has_ISR}
	The group $GL(\infty,F_2)$ has ISR.
\end{proposition}
\begin{proof}
Let $\mathcal P$ be an invariant von Neumann subalgebra of $L(GL(\infty,F_2))$. Let $E$ be the $\tau$-preserving conditional expectation onto $\mathcal{P}$. Consider the subgroup $GL(2,F_2)=\{\{g_{ij}\}_{i,j\in\mathbb N}\in GL(\infty,F_2):g_{ij}=\delta_{ij}\;\;\text{whenever}\;\;i>2\;\;\text{or}\;\;j>2\}$. It is not hard to see that for any $g\in GL(2,F_2)$ one has $\fpc(g)\subseteq GL(2,F_2)$. 

For $s=\begin{bmatrix}
	0&1\\1&0
\end{bmatrix}$
we have 
$$E\left(s\right)=\sum\limits_{g\in GL(2,F_2)}\lambda_g g$$
for some scalars $\lambda_g \in \mathbb{C}$, with $\lambda_e=0$. Let $\omega\in GL(\infty,F_2)$ be the map acting on the basis of $F_2^\infty$ by sending $e_1$ to $e_2$, $e_2$ to $e_3$, $e_3$ to $e_1$, and $e_j$ to $e_j$ for each $j>3$, \ie $\omega=\begin{bmatrix}
	0&0&1\\
	1&0&0\\
	0&1&0
\end{bmatrix}$. Observe that the action $g\to \omega g\omega^{-1}$ permutes simultaneously the rows and the columns of $g$ via the cyclic permutation $(123)$. In particular, 
$$\omega GL(2,F_2)\omega^{-1}=\{\{g_{ij}\}_{i,j\in\mathbb N}\in GL(\infty,F_2):g_{ij}=\delta_{ij}\;\;\text{whenever}\;\;i\notin\{2,3\}\;\;\text{or}\;\;j\notin\{2,3\}\}.$$ 
Since $s^{-1}E(s)\in\mathcal P'$, it commutes with $\omega E(s)\omega^{-1}=E(\omega s\omega^{-1}).$ Consider the product $$s^{-1}E(s)\cdot \omega E(s)\omega^{-1}=\sum\limits_{g\in GL(2,F_2)}\lambda_g s^{-1}g\cdot \sum\limits_{h\in GL(2,F_2)}\lambda_h \omega h\omega^{-1}.$$ Observe that when calculating the latter product, no cancellation is possible, since $GL(2,F_2)\cap \omega GL(2,F_2)\omega^{-1}=\{I\}$, where $I$ is the identity matrix. Notice that for any $A\in GL(2,F_2),B\in \omega GL(2,F_2)\omega^{-1}$ given $C=AB$ and $D=BA$ one has $C_{31}=0$ and $D_{13}=0$. 
		
Assume now that $\lambda_g\neq 0$ for some $g \in GL(2,F_2)$, $g\neq s$.  Since $s^{-1}E(s)$ is invariant with respect to conjugation by $s$, without loss of generality we may assume that $A=s^{-1}g$ has $A_{12}=1$; in other words, $g = \begin{bmatrix} a' & b' \\ c' &1 \end{bmatrix}$. As $g\neq I$, and $E$ is $^*$-preserving, there exists $h\in GL(2,F_2)$ with $h_{12}=1$ and $\lambda_h\neq 0$ (either $h=g$ or $h=g^{-1}$). Then $B=\omega h\omega^{-1}$ has $B_{23}=1$. It follows 
that $C=AB$ has $C_{13}=1$. We obtain that the Fourier expansion of $s^{-1}E(s)\cdot \omega E(s)\omega^{-1}$ contains an element $C\in GL(\infty,F_2)$ with a nonzero coefficient $\lambda_g\lambda_h$ such that $C$ does not appear in the Fourier expansion for $\omega E(s)\omega^{-1}\cdot s^{-1}E(s)$.

This contradiction shows that $E(s)=\lambda_s s$. Clearly, either $E(s)=0$ or $E(s)=s$.
		\vskip 0.1cm
\noindent ${\bf a)}$ $E(s)=s$. Since $GL(\infty,F_2)=SL(\infty,F_2)$ is simple as already mentioned above,
this implies that $E(g)=g$ for all $g\in GL(\infty,F_2)$ and so $\mathcal P=L(GL(\infty,F_2))$.\
		\vskip 0.1cm
\noindent ${\bf b)}$ $E(s)=0$. Consider the character $\chi(g)=\tau(g^{-1}E(g))$.  By Skudlarek's description of indecomposable characters in \cite{Skudlarek}, each non-regular indecomposable character on $GL(\infty,F_2)=SL(\infty,F_2)$ has the form $\chi_m(g)=2^{-m{\rm rank}(g-I)}$, ($ g \in GL(\infty,F_2)$) for some $m\in\mathbb Z_+$ and thus takes a strictly  positive value at $s$. It follows that $\chi$ is the regular character and $\mathcal P=\mathbb C 1$. 	
\end{proof}

Before we consider the general case, we shall introduce some more notation. It will be convenient to identify the algebra $L( F_2^\infty)$ with the algebra $L^\infty(\widehat{ F}_2^\infty)$, where $\widehat{F}_2$ is the dual space of $F_2$. One can write $\widehat{F}_2=\{\eta_0,\eta_1\}$, where $\eta_0$ is the trivial character and $\eta_1(z)=(-1)^z$ for $z\in F_2$. Given $i\in\{0,1\}$ and $j\in\mathbb N$ denote by $\delta_i^{(j)}\in L^\infty(\widehat{ F}_2^\infty)$ the function $\delta_i^{(j)}(x)=\delta_{x_j,\eta_i}$. Observe that for any $j$ for the $j$th basis vector $e_j\in F_2^\infty$ viewing $u_{e_j}\in L(F_2^\infty)$ as element of $L^\infty(\widehat{F}_2^\infty)$ one has:
\begin{equation}\label{eq: delta_i^j}
	\delta_0^{(j)}=\tfrac{1}{2}(1+u_{e_j}),\;\; \delta_1^{(j)}=\tfrac{1}{2}(1-u_{e_j}),\;\;\text{so}\;\;\delta_i^{(j)}=\tfrac{1}{2}(1+(-1)^iu_{e_j}),\;i\in\{0,1\}.
\end{equation} For convenience of writing elements of $L^\infty(\widehat{F}_2^\infty)$, we introduce a formal symbol $\star$ such that $\delta_\star=\id\in L^\infty(\widehat{F}_2^\infty)$. 
Note that $\{[w]: w\in \{0,1\}^k\}$ is a spanning set for $L(F_2^k)$.
For a finite vector-row $${\rm w=(w_1,w_2,\ldots,w_k)}$$ with entries from the set $\{0,1,\star\}$ denote $${\rm [w]=[w_1,\ldots,w_k]}\in L^\infty(\widehat{ F}_2^\infty),\;\; [{\rm w}]=\prod\limits_{1\leqslant j\leqslant k} \delta_{{\rm w}_j}^{(j)}=\prod\limits_{j:{\rm w}_j\neq \star} \delta_{{\rm w}_j}^{(j)}.$$ In particular, $L(F_2^k)$ embeds into $L(F_2^{k+1})$ within $L(F_2^\infty)=L^\infty(\widehat{F}_2^\infty)$ via the formula $${\rm [w_1,\ldots,w_k]}\to {\rm [w_1,\ldots,w_k,\star]}.$$

Notice that for any ${\rm w,v}\in\{0,1,\star\}^k$ one has ${\rm [w]\cdot [v]}=0$ if there exists $1\leqslant j\leqslant k$ such that simultaneously ${\rm w}_j\neq {\rm v}_j$, ${\rm w}_j\neq\star$, and ${\rm v}_j\neq\star$. If ${\rm [w]\cdot[v]}\neq 0$, then one has:
\begin{equation}\label{eq:[w]*[u]}
	{\rm [w]\cdot [v]=[z]},\;\;\text{where}\;\;{\rm z}_j={\rm w}_j,\;\;\text{if}\;\;{\rm w}_j={\rm v}_j\;\;\text{or}\;\;{\rm v}_j=\star,\;\;\text{and}\;\;{\rm z}_j={\rm v}_j\;\;\text{otherwise}.
\end{equation}
Using \eqref{eq: delta_i^j} it is straightforward to verify that for any vector-row ${\rm w}\in\{0,1\}^n$ in $L( F_2^\infty)=L^\infty(\widehat{ F}_2^\infty)$ the following holds:
\begin{equation}\label{eq: [w] expansion}
	[{\rm w}]=\tfrac{1}{2^n}\sum\limits_{v\in F_2^n}(-1)^{{\rm w}\cdot v}u_v,
\end{equation} where ${\rm w}\cdot v=\sum\limits_{1\leqslant j\leqslant n}{\rm w}_jv_j\in F_2$.

\begin{lemma}\label{lemma: g[w]/g}
	Let ${\rm w}\in\{0,1,\star\}^n$ and $g\in GL(\infty,F_2)$ be such that for each $i$ with ${\rm w}_i=\star$ the $i$-th row of $g^{-1}$ contains only one nonzero element. Then $u_g[{\rm w}]u_g^*=[{\rm w}\cdot g^{-1}]$.  Here, when calculating the entries of ${\rm w}\cdot g^{-1}$ concerning the symbol $\star$ the following rules are applied: $$\star\cdot 0=0,\;\star\cdot 1=\star+0=\star+1=\star+\star=\star.$$
\end{lemma} 
\begin{proof}
	First, let us prove the lemma in the case when ${\rm w}\in\{0,1\}^n$ and $g\in GL(n,F_2)$. By \eqref{eq: [w] expansion}, one has:
	\begin{align*}
		u_g[{\rm w}]u_{g^{-1}}&=\tfrac{1}{2^n}u_g\sum\limits_{v\in F_2^n}(-1)^{{\rm w}\cdot v}u_vu_{g^{-1}}\\&=\tfrac{1}{2^n}\sum\limits_{v\in F_2^n}(-1)^{{\rm w}\cdot v}u_{g(v)}\\&=\tfrac{1}{2^n}\sum\limits_{y\in F_2^n}(-1)^{{\rm w}\cdot g^{-1}(y)}u_y\\&=[{\rm w}g^{-1}].
	\end{align*}
	
	Next, notice that the claim of Lemma \ref{lemma: g[w]/g} is straightforward in the case when $g\in S_\infty\subseteq GL(\infty,F_2)$, \ie when $g$ has exactly one nonzero element (equal to 1) in each row and in each column. In this case $g$ acts by permuting the coordinates of ${\rm w}$.  
	
	Consider the general case. Let ${\rm w}\in\{0,1,\star\}^n$ and $g\in GL(\infty,F_2)$ satisfy the conditions of Lemma \ref{lemma: g[w]/g}. By replacing $[{\rm w}]$ with $[{\rm w}s]$ and $g$ with $s^{-1}gs$ for an appropriate $s\in S_\infty$ we can assume that ${\rm w}_i\in \{0,1\}$ for $1\leqslant i\leqslant k$ and ${\rm w}_i=\star$ for $k+1\leqslant i\leqslant n$ and $g$ is of the form $g=\begin{bmatrix}g_1&h\\0&g_2\end{bmatrix}$ for some $g_1\in GL(k,F_2)$, $g_2\in S_{n-k}$ and $h$ an $k\times (n-k)$ matrix over $F_2$. Then, factoring out the permutation $g_2$ of the coordinates $\{k+1,\ldots,n\}$ we may assume that $g_2$ is the trivial permutation (\ie equal to the identity matrix). Next, using already proven case $w\in\{0,1\}^n$, we can factor out $g_1$. It remains to treat the case $g_1=g_2=I$.
	
	Observe that ${\rm w}$ as above is a linear combination of terms of the form $u_v$, $v\in F_2^k$. Given $g$ as above with $g_1=g_2=I$ we have $g(v)=v$ and ${\rm w}g={\rm w}$. It follows that $u_gu_vu_{g^{-1}}=u_v$ for every $v\in F_2^k$ and $u_g[{\rm w}]u_{g^{-1}}=[{\rm w}]=[{\rm w}g^{-1}]$, which finishes the proof.
\end{proof}

Fix now an invariant von Neumann subalgebra $\mathcal{M}\subseteq L(G)$, and denote by $E:L(G) \to \mathcal{M}$ the unique $\tau$-preserving conditional expectation.

\begin{lemma}\label{cor: E(v)=0 or v}
Either $E(u_v)=u_v$ for all $v\in F_2^\infty$, or $E(u_v)=0$ for all $v\in F_2^\infty\setminus \{0^\infty\}$.
\end{lemma}
\begin{proof}
	It suffices to note that for a fixed $v \in F_2^\infty \setminus \{0^\infty\}$ we have, by Lemma \ref{lem:fpc(v)} and standard trace computations, $E(u_v)=\lambda_v u_v$, where $\lambda_v \in \{0,1\}$. But then if some $\lambda_v$ equals 1, we have $u_v \in \mathcal M$, and as the action of $GL(\infty, F_2)$ on $F_2^\infty \setminus \{0^\infty\}$ is transitive, we have $L(F_2^\infty) \subseteq \mathcal M$.    
\end{proof}

The next result describes another aspect of the same phenomenon. The proof is similar to that of Proposition \ref{prop: M intersect L(widetilde C(X))}.
\begin{proposition}\label{prop: M and L(Z2infty)}
	Either $\mathcal M\supseteq L(F_2^\infty)$ or $\mathcal M\cap L(F_2^\infty)=\mathbb C1$.
\end{proposition}
\begin{proof}
The arguments are similar to the proof of Proposition \ref{prop: M intersect L(widetilde C(X))}. Let $m\in\mathcal M\cap  L(F_2^\infty)$, $m\notin\mathbb C1$. Without loss of generality, $0^\infty\notin\suppe(m)$. Let $v\in\suppe(m), v\in F_2^\infty$.  Recall the Fourier expansion formula: 
$$m\delta_e=\sum\limits_{w\in F_2^\infty} m_w \delta_w,\;m_w\in\mathbb C.$$ 
We shall normalize $m$ so that $m_v=1$. 
	
Starting with $p=m$, we will repeatedly perform the following operation. Assume that $p_w\neq 0$ for some $w\neq v$. Say we choose $w\neq v$ with the largest value of $|p_w|$. Using Lemma \ref{lemma: disjoint subsets}
construct a sequence of elements $g_n\in GL(\infty,F_2)$ such that $g_n(v)=v$ but $g_n(w)$ are pairwise distinct for $n\in\mathbb N$. Let $p'$ be any weak limit point of $u_{g_n}pu_{g_n}^*$. Replacing $g_n$ with a subsequence we may assume that $u_{g_n}pu_{g_n}^*$ converges weakly to $p'$. For any $x\in F_2^\infty$ one has:
$$ p_{g_n(x)} = \langle \delta_{g_n(x)},p\delta_e\rangle = \tau(u_{g_n^{-1}} u_{x}^* u_{g_n}p) = \langle \delta_x,u_{g_n}pu_{g_n^{-1}}\delta_e \rangle\longrightarrow p'_x.$$ 
	
Next, let us order all vectors $x\in F_2^\infty$ in an arbitrary way: $$F_2^\infty=\{x_1,x_2,x_3,\ldots\}.$$ Similarly to Proposition \ref{prop: M intersect L(widetilde C(X))}, construct an increasing sequence $(n_k)_{k \in \N}$ of positive integers and a decreasing sequence of infinite subsets $S_k\subseteq\mathbb N$ such that the following holds for each $k\in \N$ 
	\begin{itemize}
		\item[$1)$] $n_k\in S_{k-1}$;
		\item[$2)$] $\min S_k>n_k$;
		\item[$3)$] either $g_j(x_k)$ does not depend on $j$ for $j\in S_k$ or $g_j(x_k)$ are pairwise distinct for $j\in S_k$.
	\end{itemize} 
Now, replace $g_n$ with the  subsequence $(g_{n_k})_{k \in \N}$. Then for any vector $x\in F_2^\infty$ the following is true: either  $g_n(x)$ becomes constant for large $n\in \N$ or  the vectors $\{g_n(x):n\geqslant N\}$ are pairwise distinct for sufficiently large $N\in \N$. Recall that $p'$ is the weak limit of $g_npg_n^*$. In particular, $p'_x=\lim_{n \to \infty} p_{g_n(x)}$ for any $x\in F_2^\infty$. We obtain in the first case above
that $p'_x=p_y$ for some $y$ which is eventually left fixed by $g_n$, and in the second case that $p'_x=0$. In particular, $p'_v=p_v,\;p'_w=0$.
	
By performing the operation $p\to p'$ repeatedly starting with $p=m$ we obtain a sequence of elements $m^{(n)}\in \mathcal M\cap L(F_2^\infty)$ such that $m^{(n)}_v=m_v=1$ for all $n$ and for any $w\neq v$ one has $m^{(n)}_w=0$ for all sufficiently large $n$. It follows that $m^{(n)}$ converges weakly to $v$. Therefore, $v\in \mathcal M$. By $GL(\infty,F_2)$-invariance, $\mathcal M\ni w$ for any $w\in F_2^\infty$, which finishes the proof.
\end{proof}

For the rest of the proof set
\begin{equation}\label{eq: s = 01 10}
	s=(12)=\begin{bmatrix}0&1\\1&0\end{bmatrix}\in GL(2, F_2)\subseteq GL(\infty, F_2)\subseteq G.
\end{equation}
Observe that $\fpc(s)\subseteq GL(2, F_2)\ltimes F_2^2$. Therefore,
\begin{equation}\label{eq:E(01_10)}E(u_s)=\sum\limits_{g\in GL(2, F_2)}u_g\cdot A_g,\;\;A_g\in L( F_2^2)=L^\infty(\widehat{ F}_2^2).\end{equation} 
Note that $A_e=0$. Indeed, consider arbitrary $v\in F_2^{\infty}$, $v \neq 0^\infty$. Then we have 
\[\tau(A_eu_v) = \tau(E(u_s) u_v) = \tau(u_sE(u_v)) =0,\]
as  by Lemma \ref{lem:fpc(v)}  $E(v)=0$ or $E(v)=v$. On the other hand $\tau(A_e)=\tau(E(u_s))=0$, so we indeed see that $A_e =0$.

The next proposition forms the key technical part of the proof of the main theorem of this section.
\begin{proposition}\label{lemma: E(01 10) for cross product}
We have  $E(s)=s\cdot A$, where $A\in  L^\infty(\widehat{ F}_2^2)$ belongs to the set $\{0,1,[0,0]+[1,1]\}$.
\end{proposition}
\begin{proof}
The ideas in the first part of the proof, where we will establish that $E(s) = s \cdot A$ for some $A \in L^\infty(\widehat{ F}_2^2)$  are similar to these used in the proof of Proposition \ref{prop:(GLinfty,F2)_has_ISR}. 

Consider the expansion in the formula \eqref{eq:E(01_10)}. As in Proposition \ref{prop:(GLinfty,F2)_has_ISR}, let $\omega=(123)\in S_\infty\subseteq GL(\infty, F_2)$. Notice that $s^{-1} E(s)=sE(s)$ commutes with $E(\omega s\omega^{-1})=\omega E(s)\omega^{-1}=\omega sE(s)s^{-1}\omega^{-1}$. Consider the product
$$sE(s)\cdot \omega sE(s)s^{-1}\omega^{-1}=\sum\limits_{g\in GL(2, F_2)}sgA_g\cdot \sum\limits_{h\in GL(2, F_2)} \omega shA_hs^{-1}\omega^{-1}.$$
\vskip 0.1cm\noindent
${\bf 1)}$ Consider the element $g=\begin{bmatrix}
	1&0\\1&1
\end{bmatrix}$. Let $A_g=\sum\limits_{i,j\in\{0,1\}}c_{ij}\cdot [i,j]$, where $c_{ij}\in\mathbb C$. We have used the spanning property of $[i,j]$, i.e., $\{[w]: w\in \{0,1\}^k\}$ is a spanning set for $L(F_2^k)$. Taking into account that $C:=sg\cdot \omega sgs^{-1}\omega^{-1}$ has $C_{13}=1$, as  $\omega sgs^{-1}\omega^{-1} = \begin{bmatrix}
	1&0&0\\0&1&1\\0&0&1
\end{bmatrix}$ and so $C$ cannot appear with a nonzero coefficient in the expansion of $\omega sE(s)s^{-1}\omega^{-1}\cdot sE(s)$ we obtain:
$$sgA_g\cdot \omega sgA_gs^{-1}\omega^{-1}=0\;\;\Longrightarrow\;\; A_g\cdot h^{-1} A_gh=0,\;\;\text{where}\;\; h=(\omega sg)^{-1}=\begin{bmatrix}
	0&0&1\\0&1&1\\1&0&0
\end{bmatrix}.$$ 
Further, substituting the formula for $A_g$ we obtain:
\begin{align*}0&=A_g\cdot h^{-1} A_gh=\sum\limits_{i,j\in\{0,1\}}c_{ij}[i,j,\star]\cdot\sum\limits_{i,j\in\{0,1\}}c_{ij}[(i,j,\star)\cdot h]\\&=
	(c_{00}[0,0,\star]+c_{01}[0,1,\star]+c_{10}[1,0,\star]+c_{11}[1,1,\star])(c_{00}[\star,0,0]+c_{01}[\star,1,1]\\& \;\;\;\;\;+c_{10}[\star,0,1]+c_{11}[\star,1,0])\\&=
	c_{00}^2[0,0,0]+c_{01}^2[0,1,1]+c_{10}^2[1,0,1]+c_{11}^2[1,1,0]+
    \\& \;\;\;\;\;+c_{00}c_{10}[0,0,1]+c_{01}c_{11}[0,1,0]+c_{10}c_{00}[1,0,0]+c_{11}c_{01}[1,1,1].\end{align*}
   We conclude that $c_{ij}=0$ for all  $i,j\in\{0,1\}$, and so $A_g=0$.
\vskip 0.1cm\noindent
${\bf 2)}$  Consider now the element $t=\begin{bmatrix}
	0&1\\1&1
\end{bmatrix}$. Let $A_t=\sum\limits_{i,j\in\{0,1\}}c_{ij}\cdot [i,j]$, where $c_{ij}\in\mathbb C$. Observe that $sE(s)$ commutes with $\omega gE(s)g\omega^{-1}=E(\omega gsg^{-1}\omega^{-1})$, where $g=g^{-1}$ is the element from part $1)$. Similarly to case $1)$ we notice that for $C:=st\cdot \omega gtg^{-1}\omega^{-1}$ one has $C_{13}=1$ and so $C$ cannot appear with nonzero coefficient in the expansion of $\omega gE(s)g\omega^{-1}\cdot sE(s)$. We obtain:
$$stA_t\cdot \omega gtA_tg\omega^{-1}=0\;\;\Longrightarrow\;\; A_t\cdot h^{-1} A_th=0,\;\;\text{where}\;\; h=(\omega gt)^{-1}=\begin{bmatrix}
	0&0&1\\0&1&0\\1&0&0
\end{bmatrix}.$$ 
Similarly to case $1)$, substituting the formula for $A_h$ we obtain:
\begin{align*}0&=A_t\cdot h^{-1} A_th=\sum\limits_{i,j\in\{0,1\}}c_{ij}[i,j,\star]\cdot\sum\limits_{i,j\in\{0,1\}}c_{ij}[(i,j,\star)\cdot h]\\&=
	(c_{00}[0,0,\star]+c_{01}[0,1,\star]+c_{10}[1,0,\star]+c_{11}[1,1,\star])(c_{00}[\star,0,0]+c_{01}[\star,1,0]\\& \;\;\;\;\;+c_{10}[\star,0,1]+c_{11}[\star,1,1])\\&=
	c_{00}^2[0,0,0]+c_{01}^2[0,1,0]+c_{10}^2[1,0,1]+c_{11}^2[1,1,1]
    \\& \;\;\;\;\;+c_{00}c_{10}[0,0,1]+c_{01}c_{11}[0,1,1]+c_{10}c_{00}[1,0,0]+c_{11}c_{01}[1,1,0].\end{align*}
    We conclude that $c_{ij}=0$ for all  $i,j\in\{0,1\}$, and so $A_t=0$.

\smallskip
 Taking into account that $E(s)=sE(s)s^{-1}$, we obtain that $A_g=0$ for all $g\neq s$, which shows the statement claimed in the first line of the proof; we shall simply write $A$ for $A_s$ from now on.
\smallskip 
 
Notice that by the invariance of $E$ the element $s$ commutes with $A$, which means that $c_{01}=c_{10}$.
Further, by Proposition \ref{prop: E properties} (3), the element $A$ lies in $\mathcal M'$. Therefore, $A$ commutes with $\omega E(s)\omega^{-1}=\omega sA\omega^{-1}$. Let $A=\sum_{i,j\in\{0,1\}}c_{ij}[i,j]$. Then we have:
 \begin{align*}
 	A&\omega s A\omega^{-1}=\omega s A\omega^{-1}A\;\;\Longleftrightarrow\;\;\omega^{-1}A\omega\cdot A=A\cdot s\omega^{-1}A\omega s^{-1}\;\;&\Longleftrightarrow\\
 	\sum  &  c_{ij}([i,j]\omega)\cdot A=A\cdot\sum c_{ij}([i,j]\omega s^{-1})\;\;&\Longleftrightarrow\\ 
 	\sum &   c_{ij}[j,\star,i]\cdot \sum c_{kl}[k,l,\star]   =\sum c_{kl}[k,l,\star]\cdot \sum c_{ij}[\star,j,i]      \;\;&
 	\Longleftrightarrow \\ \sum & c_{ik}c_{kl}[k,l,i] =\sum c_{kl}c_{il}[k,l,i]   \;\;\Longleftrightarrow\;\; c_{ik}c_{kl}=c_{kl}c_{il}\;\;&\forall\;\;k,l,i\in\{0,1\}^3.\;\;\;\;\;\;\;\;\;\;\;\;\;\;
 \end{align*}
We will now analyse several cases.

\smallskip 
\noindent $a)$ If $c_{01}\neq 0$ (equivalently, $c_{10}\neq 0$), then taking $(k,l,i)=(0,1,0)$ and $(k,l,i)=(0,1,1)$ we obtain respectively that $c_{00}=c_{01}$ and $c_{10}=c_{11}$. Since all $c_{ij}$ are equal, one has $A=c_{00}\id$. Thus, $c_{00}s=E(s)=E(E(s))=c_{00}^2s$. It follows that $c_{00}\in \{0,1\}$ and $E(s)\in \{0,s\}$.

\smallskip\noindent
 $b)$ Suppose that $c_{01}=c_{10}=0$ but $c_{00}\neq 0$ and $c_{11}\neq 0$, so that $A \neq 0$. Since $E(s)(E(s))^*=(E(s))^2=A^2\in \mathcal M$, from Corollary \ref{cor: E(v)=0 or v} we obtain that $\mathcal M\supseteq L(F_2^\infty)$. Therefore, $s\cdot ([0,0]+[1,1])\in \mathcal M$. Now we may write $E(s)=sA=s(c_{00}[00]+c_{11}[11])\in \mathcal{M}$. Then $\mathcal{M}\ni E(s)\cdot[\star1]=s(c_{00}[01])$ since $[\star1]\in\mathcal{M}$. Similarly, $\mathcal{M}\ni E(s)\cdot [\star0]=sc_{11}[00]$. Then using $\langle s-sA, [01]\rangle=0=\langle s-sA,[00]\rangle$, we can deduce that $c_{00}=c_{11}=1$.

 \smallskip \noindent $c)$ Assume now that $E(s)=c_{00}s\cdot [0,0]$, where $c_{00}\neq 0$. Conjugating by the element $t=\begin{bmatrix}
 	0&1\\1&1
 \end{bmatrix}$ we obtain that $E(g)=c_{00}g\cdot [0,0]$, where $g=\begin{bmatrix}
 	1&0\\1&1
 \end{bmatrix}$. Let $g_k=I+e_{k1}$ for $k\geqslant 2$, so that $g_2=g$. Then $g_k$ commutes with $g$, and so with $E(g)$ as well. However, for $k\geqslant 3$, $$g_k\cdot [0,0]\cdot g_k= g_k\cdot ([0,0,\star^{k-3},0]+[0,0,\star^{k-3},1])\cdot g_k=[0,0,\star^{k-3},0]+[1,0,\star^{k-3},1]\neq [0,0].$$
  We obtain a contradiction to the formula for $E(g)$, which shows that this case is impossible.

\smallskip \noindent $d)$ Finally, the case $E(s)=c_{11}s\cdot [1,1]$ can be treated similarly to the case $c)$.

\end{proof}

We shall now examine the possibilities listed in the proposition above and see that the value of $E(s)$ in fact (almost) determines the invariant subalgebra $\mathcal M$.

\begin{lemma}\label{prop: (01 10)E((01 10))=1}
 If $s E(s)=1$ then $\mathcal M=L(G)$.
\end{lemma}
\begin{proof}
	Recall that $N=\{g\in GL(\infty,F_2)\ltimes F_2^\infty:E(g)=g\}$ is a normal subgroup in $GL(\infty,F_2)\ltimes F_2^\infty$. Since $E(s)=s$, by the description of normal subgroups of $GL(\infty,F_2)\ltimes F_2^\infty$ (see Lemma \ref{lemma: normal subgroups of H}) we obtain that $N=GL(\infty,F_2)\ltimes F_2^\infty$, which finishes the proof.
\end{proof}

\begin{lemma}\label{prop: (01 10)E((01 10))=0}
 If $sE(s)=0$ then either $\mathcal M=\mathbb C 1$ or $\mathcal M=L(F_2^\infty)$.
\end{lemma}
\begin{proof}
	Consider the character $\chi(g)=\tau(gE(g^{-1})), g \in G$ (see Proposition \ref{prop: E properties} (4)). Since there are only countably many indecomposable characters on $G$ (by Theorem \ref{theorem: characters Gl times F2infty}),  from the decomposition theorem we obtain that $\chi$ has the form $\sum\limits_{i\in I}\alpha_i\chi_i$, where $I$ is some at most countable index set, $\chi_i,i\in I$, are pairwise distinct indecomposable characters, $\alpha_i>0,i\in I$, and $\sum\limits_{i\in I}\alpha_i=1$.

As $\chi(s)=0$, it follows that $\chi_i(s)=0$ for all $i\in I$. By Theorem \ref{theorem: characters Gl times F2infty}, this implies that for each $\chi_i, i\in I$, the corresponding parameter $k$ is equal to infinity. 
Examining the two characters with $k=\infty$ we conclude that $\chi(gv)=0$, and so $E(gv)=0$, for all $g\in GL(\infty, F_2)\setminus\{e\},v\in F_2^\infty$. Using Corollary  \ref{cor: E(v)=0 or v}, we obtain:
$$\mathcal M=E(L(G))=E(L(F_2^\infty))\in\{\mathbb C1,L(F_2^\infty)\}\},$$ which finishes the proof.	
\end{proof}

Before treating the last case, we will introduce the relevant exotic invariant subalgebra.

Note that in the rest of the section, we will use the notation $u_g$, $u_v$ (and not simply $g$, $v$) for elements of $L(G)$.
For $g\in GL(\infty, F_2)$ set 
\begin{equation}\label{eq: f_g}
	f_g=\tfrac{1}{\# R(g-I)}\mathbf{1}_{R(g-I)}=\tfrac{1}{\# R(g-I)}\sum\limits_{v\in R(g-I)}u_v\in L(F_2^\infty).
\end{equation} 
Clearly, $f_g$ is an orthogonal projection. Observe that $R(g-I)$ above is finite, thus, $f_g$ belongs in fact to $\mathbb C[F_2^\infty]$; note also that for any $g,h \in GL(\infty,F_2)$ we have 
\begin{equation}\label{eq: f_g commutation}
	f_gf_h=f_hf_g\leqslant f_{gh},\; f_g u_h = u_h f_{h^{-1}gh},
\end{equation} 
as for any $v \in F_2^\infty$ we have $v \in R(g-I)$ if and only if $h^{-1}v \in R(h^{-1}gh - I)$, and $R(gh-I)=R(gh-h+h-I)\subseteq R(g-I)+R(h-I)$. 

\begin{lemma}\label{prop: extra subalgebra of L(H)}
Let $\mathcal M_{exo}$ be the von Neumann subalgebra of $L(G)=L(GL(\infty, F_2)\ltimes F_2^\infty)$ generated by $L( F_2^\infty)$ and elements of the form $u_g\cdot f_g,g\in GL(\infty,F_2)$. Then $\mathcal M_{exo}$ is an exotic invariant subalgebra.	
\end{lemma}
\begin{proof}
 First, let us show that linear combinations of elements of the form $u_g f_g u_v$, $g\in GL(\infty, F_2), v\in  F_2^\infty$, are dense in $\mathcal M$. We have for any $g,h\in GL(\infty,F_2),v,w\in F_2^\infty$:
$$u_gf_gu_v\cdot u_hf_hu_w=u_{gh}\cdot f_{h^{-1}gh}f_h\cdot u_{h^{-1}(v)+w}.$$ Next, let us show that \begin{equation}\label{equation: hfh multiplication}f_{h^{-1}gh}f_h=f_{gh}\cdot A\end{equation} for some $A\in L(F_2^\infty)$. Observe that $$R(gh-I)=R((g-I)h+(h-I))\subset R(g-I)+R(h-I).$$ On the other hand, \begin{align*}
	R(h^{-1}gh-I)+R(h-I)=h^{-1}R(g-I)+R(h-I)=\\ ((h^{-1}-I)+I)R(g-I)+R(h-I)=R(g-I)+R(h-I),\end{align*} since $(h^{-1}-I)R(g-I)=-(h-I)h^{-1}R(g-I)\subset R(h-I)$. From the latter it follows that $$f_{h^{-1}gh}f_h=
    f_{gh}\cdot f_{h^{-1}gh}f_h.$$ 
    We obtain that the space $\mathcal S$ of linear combinations of the elements of the form $u_gf_gu_v$ is closed under multiplication. Since it generates $\mathcal M$ as a von Neumann algebra, $\mathcal S$ is dense in $\mathcal M$.

Further, observe that for any  $g,h\in GL(\infty, F_2)$, $v,w\in F_2^\infty$ one has:
$$u_h\cdot u_gf_gu_v\cdot u_{h^{-1}}=u_{hgh^{-1}}f_{hgh^{-1}} u_{h(v)},\;\;u_w\cdot u_gf_gu_v\cdot u_{-w} =u_gf_gu_{v+g^{-1}(w)-w}.$$ It follows that $\mathcal S$ is closed under conjugation by elements of $H$, and so is $\mathcal M$. Thus, $\mathcal M$ is an invariant subalgebra in $L(G)$.

Suppose now that $\mathcal M=L(N)$ for some normal subgroup $N$ of $G$. Since for any $g\in GL(\infty, F_2)$ the subalgebra $\mathcal M$ contains nonzero elements of the form $u_gA$, where $A\in L( F_2^\infty)$, and $\mathcal{M}$ does contain $L( F_2^\infty)$, we deduce that $N$ contains both $GL(\infty,F_2)$ and $F_2^\infty$. Thus, we would necessarily have $N=G$. 

Consider the matrix $t=I+E_{12}$, where $E_{12}$ is the matrix having a nonzero element of $ F_2$ only at $(1,2)$ coordinate (first row, second column). Let $v_0\in F_2^\infty$ be the zero vector column and $v_1\in F_2^\infty$ be the vector column having $1$ at the first place and $0$ everywhere else. Notice that $t = t^{-1}$ and $f_t=u_{v_0}+u_{v_1}$. Consider the expression $\tau(u_t(u_{v_0}-u_{v_1})u_gf_gu_v),$ $g\in GL(\infty, F_2),v\in F_2^\infty$. Clearly, it is equal to zero if $g\neq t^{-1}$. If $g=t^{-1}$ then  $$\tau(u_t(u_{v_0}-u_{v_1})u_gf_gu_v)=\tau(u_t(u_{v_0}-u_{v_1})(u_{v_0}+u_{v_1})u_{t^{-1}}u_v)=0,$$
as $v_1+v_1 = v_0$. Therefore, $\mathcal M$ cannot coincide with $L(G)$. This finishes the proof.
\end{proof}

We are ready to discuss the last outstanding case. To simplify the notations, we will identify every $g\in G$ with the operator $u_g\in L(G)$.
	\begin{lemma}\label{prop: (01 10)E((01 10)) other}
If $sE(s)=[0,0]+[1,1]$ 	then  $\mathcal {M}= \mathcal{M}_{exo}$. 
\end{lemma}
\begin{proof}
	Since $E(s)E^*(s) \in L(F_2^\infty)\setminus \mathbb C\id$, from Corollary \ref{cor: E(v)=0 or v} we derive that $\mathcal M\supseteq L(F_2^\infty)$.

Suppose that there exists an element $g\in GL(\infty,F_2)$ such that $E(u_g)\neq u_g f_g$. 
Let $n\in\mathbb N$ be such that $g\in GL(n,F_2)$. Let $(n,n+1)\in S_\infty\subseteq GL(\infty, F_2)$ be the transposition of the $n$-th and $(n+1)$-st basis vectors of $F_2^\infty$. We will use that $g^{-1}E(g)\in \mathcal M'$ and so $g^{-1}E(g)$ commutes with $$E((n,n+1))=(1,n)(2,n+1)E((1,2))(2,n+1)(1,n)=([\star^{n-1},0,0]+[\star^{n-1},1,1])(n,n+1).$$

It is not hard to see that $\fpc(g)\subseteq GL(n,F_2)\ltimes F_2^n$. Therefore, we can write 
$$g^{-1}E(g)=\sum\limits_{h\in GL(n,F_2)}h B_h,\;\;B_h\in L(F_2^n).$$ Assume that $B_h\neq 0$ for some $h\neq e$. Then $h_{ij}=1$ for some $1\leqslant i,j\leqslant n$, $i\neq j$. For $t=(1,i)(j,n)\in S(n)\subseteq GL(n,F_2)$ we have that $tht^{-1}$ has $1$ at $1$st row, $n$th column. Thus, replacing $g$ by $tgt^{-1}$, we may assume that $h_{1n}=1$. Then $C=h(n,n+1)$ has $C_{1(n+1)}=1$. Such matrix $C$ can not be written as $(n,n+1)h'$ for some $h'\in GL(n,F_2)$. Taking into account that $g^{-1}E(g)$ and $E((n,n+1))$ commute we obtain that $h B_h([\star^{n-1},0,0]+[\star^{n-1},1,1])(n,n+1)=0$. Therefore, $B_h([\star^{n-1},0,0]+[\star^{n-1},1,1])=0$.

Further, write $B_h=\sum\limits_{a\in \widehat F_2^n}c_a[a]$, $c_a\in\mathbb C$. Then $$0=B_h([\star^{n-1},0,0]+[\star^{n-1},1,1])=\sum\limits_{a\in \widehat F_2^n}c_a[a_1,a_2,\ldots,a_{n-1},a_n,a_n]$$ implies that $c_a=0$ for every $a$, and so $B_h=0$. This contradicts our assumption. Thus, $B_h=0$ for all $h\neq e$.

We thus have $E(g)=gA_g$, where $A_g\in L(F_2^n)$. Since $\mathcal M\supseteq L(F_2^n)$, we obtain: $gA_g=E(g)=E(E(g))=gA_g^2$. It follows that 
\begin{equation}\label{eq: Ag expansion}
	A_g=\sum\limits_{a\in\Omega_g}[a],\text{ where }\Omega_g\subseteq\{0,1\}^n.
\end{equation}
Consider the character $\chi(g)=\tau(g^{-1}E(g))$, $g \in GL(\infty,F_2)$. For any $n\in\mathbb N$ and any $g\in GL(n,F_2)$ one has:
\begin{equation}\label{eq: tau(Ag)}
	\chi(g)=\tau(A_g)=\frac{\#\Omega_g}{2^n}\in \mathbb Z_+/2^n.
\end{equation}
By decomposition theorem, taking into account the Skudlarek's description of indecomposable characters on $GL(\infty,F_2)$ in \cite{Skudlarek}, we obtain that there exist non-negative numbers $\alpha_i,i\in\mathbb Z_+\cup\{\infty\}$ with $\sum\alpha_{i\in\mathbb Z_+\cup\{\infty\}}=1$ such that 
$$\chi(g)=\alpha_0+\sum\limits_{k\in\mathbb N}\alpha_k\,2^{-k\,\mathrm{rank}(g-I)}+\alpha_\infty\delta_{g,I},\;\;\;\;g\in GL(\infty,F_2).$$ 
For each $n\in\mathbb N, n\geqslant 2$, fix any $g_n\in GL(n,F_2)$ such that $\mathrm{rank}(g_n-I)=n$. We obtain:
\begin{equation}\label{eq: chi(gn)-chi(gn+1)}
	\mathbb Z\supseteq 2^{n+1}(\chi(g_n)-\chi(g_{n+1}))=\alpha_1+
	\sum\limits_{k\geqslant 2}\alpha_k(2\cdot2^{(1-k)n}-2^{(1-k)(n+1)}).
\end{equation}
Taking the limit with $n\to\infty$ we get $\alpha_1\in\mathbb Z$. We will now consider two possible cases.

\smallskip\noindent
$a)$ $\alpha_1=0$. In this case, it is not hard to derive from \eqref{eq: chi(gn)-chi(gn+1)} that $\alpha_k=0$ for all $k\in\mathbb N$. Then $\chi(g)=\alpha_0$ for all $g\in GL(\infty,F_2)\setminus \{I\}$. Since $\chi(s)=1/2$ we get $\alpha_0=1/2$. Consider the element $t=\begin{bmatrix}
	0&1\\1&1
\end{bmatrix}$. Then the set $\Omega_t\subseteq\{0,1\}^2$ is $t$-invariant and consists of $4/2=2$ elements. However, $t$ has one fixed point and a period 3 orbit in $\{0,1\}^2$. This contradiction shows that the case $\alpha_1=0$ is impossible. 

\smallskip\noindent
$b)$ $\alpha_1=1$. In this case, $\chi(g)=2^{-\mathrm{rank}(g-I)}$ for all $g\in GL(\infty, F_2)$. Fix $g\in GL(\infty,F_2)$. Let $g=s_1s_2\cdots s_k$ be the factorization of $g$ according to Lemma \ref{lemma: factorization in GL}. Since $E(s)=sf_s$, we have $E(s_i)=s_if_{s_i}$ for each $1\leqslant i\leqslant k$. Moreover, $$\prod\limits_{i=1}^k E(s_i)=\prod\limits_{i=1}^k s_if_{s_i}.$$ Let us show that the latter expression coincides with $gf_g$. We have by \eqref{eq: f_g commutation}:
$$\prod\limits_{i=1}^k s_if_{s_i}=\prod\limits_{i=1}^k s_i\prod\limits_{i=1}^k f_{t_i^{-1}s_it_i},$$ where $t_i=s_{i+1}s_{i+2}\cdots s_k$, $i=1,\ldots k$. As $f_t$, $t\in GL(\infty,F_2)$, are pairwise commuting orthogonal projections, the element $P=\prod\limits_{i=1}^k f_{t_i^{-1}s_it_i}$ is an orthogonal projection. 

Notice that by \eqref{eq: f_g commutation} for each $i$, $f_{t_i^{-1}s_it_i}\geqslant \prod\limits_{j=i}^{n} f_{s_i}\geqslant f_g$, since $R(s_j-I)\subseteq R(g-I)$ for all $j=i, \ldots,n$. We now claim that $P=f_g$. First, $P\leq f_g$ since $f_sf_t\leq f_{st}$. Then $P\leq f_g$ is equivalent to saying that $t_1^{-1}R(s_1-1)+t_2^{-1}R(s_2-1)+\cdots+t_{k-1}^{-1}R(s_{k-1}-1)+R(s_k-1)\subseteq R(g-1)$. Our goal is to show that equality holds. To show this, we may treat $F_2^{\infty}$ as an inner product space. So we just take any vector $\xi$ which is perpendicular to the sum of subspaces in the LHS, we aim to show $\xi$ is also orthogonal to $R(g-1)$ on the RHS. Translate this condition using inner product, we have $$(s_1-1)t_1\xi=0, (s_2-1)t_2\xi=0,\cdots,(s_k-1)\xi=0.$$ It is not hard to check that these identities actually imply that $\xi$ is invariant under all $s_1,\cdots, s_k$. So $g^{-1}\xi=\xi$, equivalently, $\xi$ is perpendicular to $R(g-I)$. Thus, $P\geqslant f_g$. 

Since $f_g\in L(F_2^\infty)\subseteq\mathcal M$ (see the beginning of the proof), we have that $gf_g=gP\cdot f_g\in \mathcal M$. Observe that 
$$\tau((g-gf_g)(g-gf_g)^\ast)=1-\tau(f_g)=1-2^{-\mathrm{rank}(g-I)}=1-\tau(g^{-1}E(g))=\tau((g-E(g))(g-E(g))^*).$$ Since $E$ is a projection onto $\mathcal M$ with respect to $\|\cdot\|_2$, we obtain that $E(g)=gf_g$. This implies that $\mathcal M$ coincides with the algebra $\mathcal{M}_{exo}$.
\end{proof}

\begin{remark}
Note that the proof of the above lemma shows in particular that for $\mathcal{M} = \mathcal{M}_{exo}$ we have $E(g) = gf_g$ for every $g  \in GL(\infty,F_2)$. 
\end{remark}

The next result summarizes the contents of this section and, in particular, establishes Theorem \ref{thm:B}.

\begin{corollary}
	The von Neumann algebra $L(\Aut_\fin (\Z_2^\infty)\ltimes \Z_2^\infty )$ contains four distinct invariant von Neumann subalgebras: $L(\Aut_\fin (\Z_2^\infty)\ltimes \Z_2^\infty )$, $\C 1$, $L(F_2^\infty)$ and $\mathcal{M}_{exo}$, the first three of which correspond to normal subgroups of $\Aut_\fin (\Z_2^\infty)\ltimes \Z_2^\infty $.
\end{corollary}
\begin{proof}
Follows immediately from Proposition \ref{lemma: E(01 10) for cross product} and  Lemmas \ref{prop: (01 10)E((01 10))=1}-\ref{prop: (01 10)E((01 10)) other}.
\end{proof}
\section{Invariant Subalgebras of the classical lamplighter group}
\label{Section: Lamplighter}
In this section, we focus our attention on the classical lamplighter group $G=\mathbb{Z}_2\wr\mathbb{Z}=(\oplus_{\mathbb{Z}}\Z_2)\rtimes \mathbb{Z}$. We aim to study all $G$-invariant von Neumann subalgebras in $L(G)$. First, note that as $G$ is a metabelian group with $\mathbb{Z}$ as the quotient, we have a short exact sequence of groups
$$0\rightarrow \oplus_{\mathbb{Z}}\Z_2\rightarrow G\rightarrow \mathbb{Z}\rightarrow 0$$
Thus every subgroup of $G$ is isomorphic (abstractly as a group) to  $H\rtimes k\mathbb{Z}$ for some subgroup $H\subset \oplus_{\mathbb{Z}}\Z_2$ and some $k\in\mathbb{Z}$.  This does not mean that the abstract isomorphism is compatible with the natural inclusions $H\subset \oplus_{\mathbb{Z}}\Z_2\subset G$ and  $k\mathbb{Z}\subset \Z \subset G$, even if the subgroup in question is normal (see an example below in Remark \ref{remark: strange normal subgroups}).
For more information on the subgroup structure of $G$, see \cite{GriKra}. 

Given a p.m.p.\ action $(X,\mu)$ and a $\mathbb{Z}$-factor map $\pi: (X,\mu)\rightarrow (Y,\nu)$, 
let $\mu=\int_Y\mu_yd\nu(y)$ denotes the measure disintegration w.r.t.\ $\pi$. Set $X\times_YX=\{(x_1,x_2)\in X\times X: \pi(x_1)=\pi(x_2)\}$ and $\theta=\int_{X}\mu_y\times\mu_yd\nu(y)$. 
\begin{definition}\label{def: relatively weakly mixint}
    We say that $\pi$ is relatively weakly mixing if the action $\mathbb{Z}\curvearrowright(X\times_YX,\theta)$ is ergodic. 
\end{definition}\noindent 
For more details on relative weak mixing, see \cite[\S~2.5]{Furman_IMRN} and the references therein.

The following lemma is essentially a simple version of \cite[Lemma 3.2]{Furman_IMRN}, which should be compared to Popa’s original \cite[Lemma 2.11]{Popa_cocycle}. We decided to include a proof for completeness. 

\begin{lemma}\label{lem: locate elements using relative weak mixing}
Let $\langle s\rangle=\mathbb{Z}\curvearrowright (X,\mu)$ and $\mathbb{Z}\curvearrowright (Y,\nu)$ be two p.m.p.\ actions and $\pi: (X,\mu)\rightarrow (Y,\nu)$ be a factor map which is relatively weakly mixing. Denote the action of $\langle s\rangle$ by $\sigma : \Z \to \textup{Aut}(X, \mu)$.  If $v\in L^{\infty}(X,\mu)$ is a unitary element such that $\sigma_{s^n}(v)v^*\in L^{\infty}(Y,\nu)$ for all $n\in\mathbb{Z}$, then $v\in L^{\infty}(Y,\nu)$. 
\end{lemma}
\begin{proof}
Since $\sigma_{s^n}(v)v^*\in L^{\infty}(Y,\nu)$, we deduce  that
$[\sigma_{s^n}(v)v^*](x_1)=[\sigma_{s^n}(v)v^*](x_2)$
 for $\theta$-a.e. $(x_1, x_2)\in X\times_YX$, i.e. $v(s^{-n}x_1)v^*(x_1)=v(s^{-n}x_2)v^*(x_2)$, or equivalently
 \begin{align*}
 v(x_1)v^*(x_2)=v(s^{-n}x_1)v^*(s^{-n}x_2)~\text{for  
 $\theta$-a.e. $(x_1, x_2)\in X\times_YX$}.
 \end{align*}
 Set $\eta(x_1,x_2)=v(x_1)v^*(x_2)$ for  $\theta$-a.e. $(x_1,x_2)\in X\times_YX$. Then the above implies $\eta$ is $\mathbb{Z}$-invariant and hence must be a constant function by the relatively weakly mixing assumption. Since $v$ is a unitary, we deduce that $\eta(x_1,x_2)=v(x_1)v^*(x_2)=c\in\mathbb{T}$ for  $\theta$-a.e. $(x_1,x_2)\in X\times_YX$. By swapping the two coordinates, we also have $c=\eta(x_2,x_1)=v(x_2)v^*(x_1)=\overline{\eta(x_1,x_2)}=\overline{c}$. Hence $c=\pm 1$. 

If $c=-1$, then  for  $\theta$-a.e. $(x_1,x_2),(x_2,x_3),(x_3,x_1)\in X\times_YX$ we have:
 $$v(x_1)=-v(x_2)=-(-v(x_3))=v(x_3)=-v(x_1).$$ Thus $v=0$, a contradiction to the unitary assumption. Thus $c=1$ and hence $v(x_1)=v(x_2)$ for  $\theta$-a.e. $(x_1,x_2)\in X\times_YX$; equivalently, $v\in L^{\infty}(Y,\nu)$.   
\end{proof}
Let $A=\oplus_{\mathbb{Z}}\Z_2$. Note that $L(A)\cong L^{\infty}((\{0,1\}^{\mathbb{Z}}, \{\frac{1}{2},\frac{1}{2}\}^{\mathbb{Z}}))$. Therefore, one might suspect  that every $G$-invariant von Neumann subalgebra of $L(G)$ is of the form $L^{\infty}(Y,\nu)\rtimes k\mathbb{Z}$ for some $\mathbb{Z}$-factor $(Y,\nu)$ of the Bernoulli shift. It turns out that this is true in the case when the corresponding factor map is relatively weakly mixing, but not in general. 

Before we formulate the corresponding result, we need one more lemma. It can be deduced directly from the classical theorem of Ornstein that every non-trivial factor of a Bernoulli shift is still measurably conjugate to a  Bernoulli shift (\cite{orn}). We provide below a direct proof.

\begin{lemma} Every factor of the Bernoulli shift $\mathbb{Z}\curvearrowright (\{0,1\}^{\mathbb{Z}}, \{\frac{1}{2},\frac{1}{2}\}^{\mathbb{Z}})$ is either trivial or essentially free. 
\end{lemma}
\begin{proof}
First, we prove that any factor $\mathbb{Z}\curvearrowright (Y,\nu)$ is either essentially free or a transitive map on a finite space, i.e.\ $[\mathbb{Z}\curvearrowright (Y,\nu)]\cong [\mathbb{Z}\curvearrowright \Z_n]$ for some $n\geq 1$. Indeed, assume that $\mathbb{Z}\curvearrowright(Y,\nu)$ is not essentially free. Write $\mathbb{Z}=\langle s\rangle$. 
Then for some $n\geq 1$ we have $\nu(\text{Fix}(s^n))>0$, where $\text{Fix}(s^n)=\{y\in Y: s^ny=y\}$. Note that since $\mathbb{Z}$ is abelian, $\text{Fix}(s^n)$ is $\mathbb{Z}$-invariant. As $\mathbb{Z}\curvearrowright (\{0,1\}^{\mathbb{Z}}, \{\frac{1}{2},\frac{1}{2}\}^{\mathbb{Z}})$ is the Bernoulli shift, it is ergodic. Since ergodicity passes to factors, we have $\nu(\text{Fix}(s^n))=1$, i.e. $s^ny=y$ for $\nu$-a.e.\ $y\in Y$.

Note that $(Y,\nu)$ has no diffuse part. Indeed, otherwise, we could pick any $Z\subset Y$ with $0<\nu(Z)<\frac{1}{n}$. Then $W:=\cup_{i=0}^{n-1}s^iZ$ would be $\mathbb{Z}$-invariant and such that $0<\nu(W)<1$, which would contradict the ergodicity of $\mathbb{Z}\curvearrowright (Y,\nu)$. Pick any atom $\{y_0\}$ in $Y$. Then $Y=\cup_{i=0}^{n-1}\{s^iy_0\}$ and clearly this shows that $\mathbb{Z}\curvearrowright (Y,\nu)$ is measurably conjugate to the transitive action $\mathbb{Z}\curvearrowright\mathbb{Z}_n$.

Second, Bernoulli shifts are mixing, and this property passes to non-trivial factor maps. Thus if $[\mathbb{Z}\curvearrowright (Y,\nu)]\cong [\mathbb{Z}\curvearrowright \mathbb{Z}_n]$, then $n=1$ and the action is trivial.
\end{proof}

The following is the main result of this section.
\begin{theorem}
\label{thm:lamplighter}
Let $G=\mathbb{Z}_2\wr\mathbb{Z}=(\oplus_{\mathbb{Z}}\Z_2)\rtimes \mathbb{Z}$ be the classical lamplighter group. 
Let $\mathcal{P}\subseteq L(G)$ be a $G$-invariant von Neumann subalgebra. Then the following hold true:
\begin{itemize}
\item[(i)] there exists a normal subgroup $N\lhd G$ such that $\mathcal{P}=(\mathcal{P}\cap L(A))\vee L(N)$, i.e.\ $\mathcal{P}$ is generated as a von Neumann algebra by $\mathcal{P}\cap L(A)$ and $L(N)$; 
\item[(ii)] $\mathcal{P}\cap L(\mathbb{Z})=L(k\mathbb{Z})$ for some $k\in\mathbb{Z}$;
\item[(iii)] if the factor map $\pi: (\widehat{A},\text{Haar})\rightarrow (Y,\nu)$ is relatively weakly mixing, where $L^{\infty}(Y,\nu)\cong \mathcal{P}\cap L(A)$, then $\mathcal{P}=L^{\infty}(Y,\nu)\rtimes k\mathbb{Z}$.
\end{itemize}
\end{theorem}
\begin{proof}
First, note that $\mathcal{P}\cap L(A)$ is a $G$-invariant von Neumann subalgebra of $$L(A)\cong L^{\infty}\left(\left(\{0,1\}^{\mathbb{Z}}, \{\frac{1}{2},\frac{1}{2}\}^{\mathbb{Z}}\right)\right),$$ and as such it corresponds to a factor of the Bernoulli shift $\mathbb{Z}\curvearrowright (\widehat{A},\text{Haar})\cong (\{0,1\}^{\mathbb{Z}}, \{\frac{1}{2},\frac{1}{2}\}^{\mathbb{Z}})$. Write $L^{\infty}(Y,\nu)=\mathcal{P}\cap L(A)$ and 
consider the corresponding factor map.

As usual set $E: L(G)\twoheadrightarrow \mathcal{P}$ be the trace-preserving conditional expectation onto $\mathcal{P}$ and set $\phi(g):=\tau(g^{-1}E(g))$ for all $g\in G$. 
Set $N=\{g\in G: \phi(g)=1\}$, recalling Proposition \ref{prop: E properties}. 
In particular, $g\in N$ implies $g\in \mathcal{P}$, i.e.\ $L(N)\subseteq \mathcal{P}$. It is also clear that $N$ is a normal subgroup of $G$.

We split the proof into two cases.

\smallskip \textbf{Case 1}. $\mathbb{Z}\curvearrowright (Y, \nu)$ is the trivial action.
\smallskip

In this case, we shall prove that $\mathcal{P}=\mathbb{C}$, $N=\{e\}$ and $k=0$. In fact, we will show that $\phi(g)=0$ for all $e\neq g\in G$. 
First, note that for each $a, b\in A$, we have $bE(a)b^{-1}=E(bab^{-1})=E(b)$ and thus $E(a)\in L(A)'\cap L(G)=L(A)$ (where we use the essential freeness of the action). Hence, actually $E(a)\in L(A)\cap \mathcal{P}=\mathbb{C}$. Thus $E(a)=0$ for all $e\neq a\in A$ and so $\phi(a)=0$. 

Fix now $n\in\mathbb{Z}$ and $e\neq a\in A$. Take any $k\in\mathbb{Z}$ such that $k>\max\{|j-i|: i,j\in\text{supp}(a)\}$. Then note that
$\sigma_{s^{km}}(a)\neq a$ for all $0\neq m\in\mathbb{Z}$. Let $t_m:=s^m(as^n)s^{-m}=\sigma_{s^m}(a)s^n$. Then $t_m^{-1}t_{m'}=s^{-n}\sigma_{s^m}(a)^{-1}\sigma_{s^{m'}}(a)s^n$. Thus $\phi(t_m^{-1}t_{m'})=\phi(\sigma_{s^m}(a)^{-1}\sigma_{s^{m'}}(a))=\phi(a^{-1}\sigma_{s^{m'-m}}(a))=0$ for all $m'\neq m\in \mathbb{Z}$ since $e\neq a^{-1}\sigma_{s^{m'-m}}(a)\in A$. Thus by the well-known lemma on characters, see e.g.\ \cite[Lemma 2.7]{dudko2024character}, we deduce that $\phi(as^n)=\phi(t_m)=0$ for all $m$. 

Finally for any $0\neq n\in\mathbb{Z}$ note that for any mutually distinct elements $\{a_i: i\in \N\}\subset A$, we have $\sigma_{s^n}(a_i^{-1}a_j)\neq a_i^{-1}a_j$ for all $i, j \in \N, i\neq j$. Thus $\phi(s^n)=0$ by a similar argument as above by noticing that $a_is^na_i^{-1}=a_i\sigma_{s^n}(a_i^{-1})s^n$ and $[a_j\sigma_{s^n}(a_j^{-1})s^n][a_i\sigma_{s^n}(a_i^{-1})s^n]^{-1}=(a_i^{-1}a_j)\sigma_{s^n}(a_i^{-1}a_j)^{-1}$. Therefore, $\phi\equiv \delta_e$ and hence $\mathcal{P}=\mathbb{C}$.

\smallskip \textbf{Case 2}. $\mathbb{Z}\curvearrowright (Y,\nu)$ is essentially free.

\smallskip
In this case, we have $L^{\infty}(Y,\nu)'\cap L(G)=L(A)\cong L^{\infty}(\{0,1\}^{\mathbb{Z}}, \{\frac{1}{2},\frac{1}{2}\}^{\mathbb{Z}})$ (see e.g. \cite[Lemma 1.6]{packer}). We briefly review the proof. Pick any $f \in L^{\infty}(Y,\nu)'\cap L(G)$, with the Fourier type decomposition $f=\sum_{n\in\mathbb{Z}}f_ns^n$, where $f_n\in L(A)$, $n \in \Z$. Then $\xi f=f\xi$ for any $\xi\in L^{\infty}(Y,\nu)$ shows that $f_n(\xi-\sigma_{s^n}(\xi))=0$, hence also $P(f_n^*f_n)(\xi-\sigma_{s^n}(\xi))=0$ for all $n\in\mathbb{Z}$, where $P: L(A)\rightarrow L^{\infty}(Y)$ is the (faithful) conditional expectation. Since $\mathbb{Z}\curvearrowright (Y,\nu)$ is essentially free, this implies that $f_n=0$ for all $n\neq 0$. Indeed, assume that for some $n\neq 0$, we have $f_n\neq 0$; equivalently, $P(f_n^*f_n)\neq 0$, i.e.\ there exists some $A\subseteq Y$ such that $\nu(A)>0$ and $P(f_n^*f_n)(x)\neq 0$ for all $x\in A$. Then by \cite[Proposition 2.4]{kerrli_book}, we may find some $B\subseteq A$ with $\nu(B)>0$ and $s^nB\cap B=\emptyset$. Set $\xi=\chi_B$, the characteristic function on $B$, then for a.e.\ $x\in B$, we get $0=P(f_n^*f_n)(x)(\xi-\sigma_{s^n}(\xi))(x)=P(f_n^*f_n)(x)\neq 0$, a contradiction.

Note that for each $g \in G$ we have $g^{-1}E(g)\in L^{\infty}(Y,\nu)'\cap L(G)=L(A)$, so that we may write $E(g)=gp_g$ for some $p_g\in L(A)$  for all $g\in G$. Hence using the bimodule property, we have that $g^nE(g)g^{-n}=E(g)$ and hence $p_g=\sigma_{g^n}(p_g)$ for all $g\in G$. In other words, $p_g$ is invariant under the action of the cyclic group $\langle g\rangle$. Note that for all $g\in G\setminus A$, as an algebraic action $\langle g \rangle\curvearrowright (\widehat{A},\text{Haar})$ is ergodic by \cite[Lemma 1.2 and Theorem 1.6]{schmidt_book} (see also \cite[Theorem 3.1]{lps}) since $\sharp\{g^n\cdot v=\sigma_{g^n}(v): n\in\mathbb{Z}\}=\infty$ for all $0\neq v\in A$. 
Therefore $p_g\in\mathbb{C}$ for all $g\in G\setminus A$.  It now follows that for any $g\in G\setminus A$, $$gp_g=E(g)=E(E(g))=E(gp_g)=E(g)p_g=gp_g^2,$$ so that $p_g=0$ or $1$; in other words $E(g) = g$ or $E(g) =0$. 
This implies that $\mathcal{P}\cap L(\mathbb{Z})=L(N\cap \mathbb{Z})=L(k\mathbb{Z})$ for a certain $k \in \Z$. 

Next, we show that $\mathcal{P}=L^{\infty}(Y, \nu)\vee L(N)$. The inclusion \say{$\supseteq$} clearly holds true since $\mathcal{P}$ contains both $L^{\infty}(Y, \nu)$ and $L(N)$. For the \say{$\subseteq$} direction, note that for any $g\in G$, either $g\in A$ or $g\in G\setminus A$. If $g\in A$, then $E(g)= g p_g\in L(A)$ and thus $E(g)\in \mathcal{P}\cap L(A)=L^{\infty}(Y, \nu)$. If $g\in G\setminus A$, then $E(g)=0$ unless $E(g)=g$ and thus $g\in N$. This shows that $E(g)\in L^{\infty}(Y)\vee L(N)$ for all $g\in G$. Since $\mathcal{P}$ is generated as a von Neumann algebra by $\{E(g): g\in G\}$, we deduce that $\mathcal{P}\subseteq L^{\infty}(Y, \nu)\vee L(N)$. Hence, we have shown that $\mathcal{P}=L^{\infty}(Y, \nu)\vee L(N)$.
\smallskip

We are left to show that (iii) holds true. Notice that $L^{\infty}(Y, \nu)\rtimes k\mathbb{Z}\subseteq \mathcal{P}$. By (i), it suffices to show that for any $v\in A$ and $n\in\mathbb{Z}$, if $vs^n\in N$, then $vs^n\in L^{\infty}(Y, \nu)\rtimes k\mathbb{Z}$.

Subcase 1. $n=0$. Then $v\in N$ implies that $v\in 
\mathcal{P}\cap L(A)=L^{\infty}(Y, \nu)$.

Subcase 2. $v=e$. Then $s^n\in \mathcal{P}\cap L(\mathbb{Z})=L(k\mathbb{Z})$ by (ii).

Subcase 3. $v\neq e$ and $n\neq 0$. 

For any $s^{\ell}\in\mathbb{Z}$, we have $\sigma_{s^{\ell}}(v)s^n=s^{\ell}(vs^n)s^{-\ell}\in \mathcal{P}$. Hence $\sigma_{s^{\ell}}(v)v^{-1}=(\sigma_{s^{\ell}}(v)s^n)(vs^n)^{-1}\in \mathcal{P}\cap L(A)=L^{\infty}(Y, \nu)$. Notice that $v\in A\subset L(A)\cong L^{\infty}(\widehat{A})$ is a unitary, we may apply Lemma \ref{lem: locate elements using relative weak mixing} to conclude that $v\in L^{\infty}(Y, \nu)\subset \mathcal{P}$. Thus $E(v)=v$, i.e.\ $v\in N$. Therefore, $s^n=v^{-1}(vs^n)\in N$ and thus $s^n\in L(N)\cap L(\mathbb{Z})\subset \mathcal{P}\cap L(\mathbb{Z})=L(k\mathbb{Z})$ by (ii). This implies that $k\mid n$, i.e.\ $s^n\in k\mathbb{Z}$, so that finally $vs^n\in L^{\infty}(Y, \nu)\rtimes k\mathbb{Z}$.
\end{proof}

\begin{remark} \label{remark: strange normal subgroups}
We remark that the relatively weakly mixing assumption could not be dropped. Indeed, take any non-trivial normal subgroup $N\lhd G$ such that $N$ is not a semidirect product inside $G$. 
For example, we may take $N$ to be the normal subgroup generated by $\delta_0s$, where $\delta_0\in A= \oplus_{\mathbb{Z}}\frac{\mathbb{Z}}{2\mathbb{Z}}$ takes value $1\in \frac{\mathbb{Z}}{2\mathbb{Z}}$ at coordinate 0 and $0\in \frac{\mathbb{Z}}{2\mathbb{Z}}$ at other coordinates. 
Then $\mathcal{P}:=L(N)\neq L^{\infty}(Y. \nu)\rtimes k\mathbb{Z}$ for any factor $(Y, \nu)$ of $\hat{A}$ and any $k \in \Z$. To see this, first note that $N\neq (N\cap A)\rtimes k\mathbb{Z}$. Indeed, suppose that we do have equality above. Then $k=1$ since $\pi(N)=\mathbb{Z}$, where $\pi: G\rightarrow G/A=\mathbb{Z}$ is the natural surjective homomorphism. Next, note that $N\cap A=\{a\in A: ~\sharp\text{supp}(a)~\text{is even}\}$, and on the other hand $\delta_0=(\delta_0s)s^{-1}\in N\cap A$, which gives us a contradiction. Finally, suppose that $L(N)=L^{\infty}(Y,\nu)\rtimes k\mathbb{Z}$. Then $L(N\cap A)=\mathcal{P}\cap L(A)=L^{\infty}(Y,\nu)$. So $L(N)=L(N\cap A)\rtimes \mathbb{Z}$, equivalently, $N=(N\cap A)\rtimes \mathbb{Z}$, which again is a contradiction. Therefore, according to part (iii), we know that $\mathbb{Z}\curvearrowright \widehat{A}\rightarrow Y$ is not relatively weakly mixing for this example. 
\end{remark}

The theorem above does not make it immediately clear whether the lamplighter group does have the ISR property. Below we show this is not the case.

\begin{proposition}\label{proposition: lamplighter groups do not have ISR}
The group $G=(\oplus_{\mathbb{Z}}\Z_2)\rtimes \mathbb{Z}$ does not have the ISR property. 
\end{proposition}
\begin{proof}
On the one hand, recall that by Sinai's celebrated theorem in \cite{Sinai}, for any number $r\in (0,log2)$, there exists a $\mathbb{Z}$-factor map $(\widehat{A},\text{Haar})\twoheadrightarrow (Y,\nu)$ such that the entropy $h(\mathbb{Z}\curvearrowright (Y,\nu))=r$. On the other hand, $A=\oplus_{\mathbb{Z}}\Z_2$ contains only countably many $\mathbb{Z}$-invariant subgroups; equivalently, there are only countably many normal subgroups of  $G$ contained in $A$. Indeed, this follows since as $G$ is a finitely generated metabelian group, every  subgroup as above is normally finitely generated  by \cite[Lemma 6.4]{BBDZ}. Therefore, there exists some $r\in (0,\log 2)$ such that $\mathcal{P}:=L^{\infty}(Y,\nu)$ is not of the form $L(B)$ for any subgroup $B\subseteq G$, where $\mathbb{Z}\curvearrowright (Y,\nu)$ is a factor of $\mathbb{Z}\curvearrowright (\widehat{A},\text{Haar})$ with entropy $r$. This shows that $G$ does not have the ISR property. 
\end{proof}

\section{Remarks on invariant subalgebras for the generalised wreath product $S_\infty \ltimes \Z_2^\infty$}
\label{sec:wreath}

In the final section, we consider the apparently simplest semidirect product of the type possible (among the ones considered in this paper). To be precise, we examine $S_\infty \ltimes \Z_2^\infty$, which can also be viewed as the (generalised) wreath product $\mathbb{Z}_2 \wr S_\infty$. Here, of course, $S_\infty$ denotes the group of all finite permutations acting on $\N$, viewed as the index set for the direct product $\Z_2^\infty = \bigoplus_{\N} \Z_2$. Let us then set
\[ G = S_\infty \ltimes \Z_2^\infty. \]

For $n\in\mathbb N$ let $\mathbb Z_2^{(n)}$ be the copy of $\mathbb Z_2$ inside $\mathbb Z_2^\infty\subseteq G$ sitting at the $n$-th coordinate. Further set $$S_n=\{s\in S_\infty: s(i)=i\text{ for all }i>n\},\;\;S_{n\infty}=\{s\in S_\infty:s(i)=i\text{ for all }i\leqslant n\}.$$ It is not hard to see that for $n\neq 2$ the subgroup $S_{n\infty}$ is the centralizer of $S_n$ in $S_\infty$.

Again, let us fix an invariant von Neumann subalgebra $\mathcal{M}\subseteq L(G)$, and denote by $E:L(G) \to \mathcal{M}$ the unique $\tau$-preserving conditional expectation. We will first establish some general properties of the pair $(\mathcal M, E)$.
\begin{proposition}
Let $l \in \N$. Then $E(L(\mathbb Z_2^{(l)}))\subseteq L(\mathbb Z_2^{(l)})$.
\end{proposition}
\begin{proof}
	Proposition \ref{prop: E properties} (2) implies that it is sufficient to consider the case $l=1$. For any $z\in \mathbb Z_2^{(1)}\setminus \{e\}$ one has $C_G(z)=\langle S_1(\infty),Z_2^\infty\rangle$. It follows that $\fpc(z)=\mathbb Z_2^{(1)}$. Applying Lemma \ref{lemma: fpc application} finishes the proof.
\end{proof}
For $k \in \N$ and $z_1,\ldots, z_k\in \mathbb Z_2$ we will denote by $(z_1,\ldots,z_k)$ the element $z_1^{(1)}z_2^{(2)}\cdots z_k^{(k)}\in\mathbb Z_2^\infty$. These should not be confused with permutations (which are written without commas).
\begin{proposition}\label{prop:E(L(Z_m^2))}
	One has $E(L(\mathbb Z_2^2))\subseteq L(\mathbb Z_2^2)$.
\end{proposition}
\begin{proof}
	Let $z_1,z_2\in\mathbb Z_2$ and $z=(z_1,z_2)\in\mathbb Z_2^2<\mathbb Z_m^\infty$. Clearly, $\fpc(z)\subseteq\langle \mathbb Z_m^2,(12)\rangle$. Using Lemma \ref{lemma: fpc application} we deduce that $E(z)=(12)A+B$ for some $A,B\in L(\mathbb Z_2^2)$.  Observe that $E(z)$ commutes with $\mathbb Z_2^2$. In particular, $1^{(1)}(12)A=(12)A1^{(1)}$, and so $A1^{(1)}=A1^{(2)}$. Writing $A=\sum\limits_{i,j\in\mathbb Z_2}a_{ij}i^{(1)}j^{(2)}$ we obtain that $a_{(i-1)j}=a_{i(j-1)}$ for all $i,j=0,1$. Therefore, there exist numbers $c_0, c_1\in\mathbb C$ such that $a_{ij}=c_{i+j}$ for all $i,j\in\mathbb Z_2$.
	
	Further, $$E((0,-z_1,-z_2))=(123) E((-z_1,-z_2))(123)^{-1}=
	(123) (A^* (12) + B^*) (123)^{-1} =
	\widetilde A   (23)+\widetilde B,$$
	where $\widetilde A=(123)A^*(123)^{-1},\;\widetilde B=(123)B^*(123)^{-1}$. 
	
	Since $E((0,-z_1,-z_2))$ commutes with $E((z_1,z_2)) \in L(\mathbb{Z}_2^2)$, we have:
	$$((12)A+B)(\widetilde A (23)+\widetilde B)=((23)\widetilde A+\widetilde B)((12)A+B).$$
	Considering the coefficient in the above equation at $(123)=(12)(23)$ 
	(which appears only on the left) we obtain that $A \widetilde A =0$. This can be written as:
	$$\sum\limits_{i\in\mathbb Z_2}c_i\sum\limits_{k\in\mathbb Z_2}(k,i-k,0)\cdot \sum\limits_{j\in\mathbb Z_2}\overline{c_{-j}}\sum\limits_{l\in\mathbb Z_2}(0,j-l,l)=0.$$
	In the above formula the coefficient at $(0,0,0)$ corresponds to $k=l=0,i+j=0$ and gives: $\sum_{i \in \Z_2} |c_i|^2=0$. This implies that $c_0=c_1=0$ and thus $E(z)=B\in L(\mathbb Z_2^2)$.
\end{proof}

We will now consider the action of $E$ on permutations, beginning with the basic transposition.
\begin{proposition}\label{prop: E(12)}
	One has $E((12))=(12)P$ for some self-adjoint element $P \in  L(\mathbb Z_2^{(1)}\times \mathbb Z_2^{(2)})$ such that $E(P)=P^2.$ 
\end{proposition}
\begin{proof} Let $\mathcal A=L(\mathbb Z_2^2)$ and $\mathcal S=(12)\mathcal A$. By Proposition \ref{prop:E(L(Z_m^2))} one has $E(\mathcal A)\subseteq\mathcal A$. Clearly, $\tau$ is zero on $\mathcal S$. For any $z\in \mathbb Z_m^2$
	one has $C_G((12)z)\supseteq S_{2\;\infty}.$ Therefore, $\fpc((12)z)\subseteq \langle \mathbb Z_m^{(1)}\times \mathbb Z_m^{(2)},(12)\rangle$. By Lemma \ref{lemma: fpc application}, one has $E((12)z)\in\mathcal S+\mathcal A$. Thus, the conditions of Proposition \ref{prop: E(S) subset S} are satisfied. Therefore, $E((12)L(\mathbb Z_2^2))\subseteq (12)L(\mathbb Z_2^2)$. In particular, $E((12))=(12)P$ for some $P\in L(\mathbb Z_2^2)$.
	
	Further, since $E((12))$ is self-adjoint, using Proposition \ref{prop: E properties}, we obtain:
	$$P(12)=(12)E((12))(12)=E((12))=E((12))^*=P^*(12).$$ Therefore, $P=P^*$. Now, $(12)E((12))=P$. Applying $E$ on both sides, and taking into account that $P=P^*$ commutes with $(12)$, we get:
	\[P^2=(12)P(12)P=E((12))^2=E((12)E((12)))=E(P).\] 
	\end{proof}

\begin{remark} In fact, by induction  we obtain:
	\[E(P^{2k-1})=E(P^{2k})=P^{2k}, \;\;\; k \in \N.\] Observe also that $P$ can be viewed as an operator of multiplication by a function from $L^\infty(\widehat \Z_2^2)$.\end{remark}

Note that if $\mathcal{M} = L(N)$ for a normal subgroup $N \trianglelefteq G$, then  $P$ must be either $0$ or $1$. We will however now present two examples of invariant subalgebras in which first $P$ is a non-trivial projection, and then $P=0$, but the algebra $\mathcal {M}$ does not arise from a normal subgroup.

\medskip

Fix a non-trivial orthogonal projection $Q\in L(\mathbb Z_2)$. For $n\in\mathbb N$ let $Q^{(n)}$ be the isomorphic copy of $Q$ inside $L(\mathbb Z_2^{(n)})$. For a subset $A\subset\mathbb N$ set $Q^A=\prod\limits_{n\in A}Q^{(n)}$. Let $\mathcal M_Q$ be the von Neumann subalgebra of $L(G)$ generated by the operators from the set \begin{equation}\label{equation: definition of M_P}
	\{s\cdot P^{\suppg(s)}:s\in S_\infty\}\cup L(\mathbb Z_2^\infty).\end{equation}
\begin{proposition}
	The algebra $\mathcal M_Q$ is an exotic invariant subalgebra of $L(G)$.   
\end{proposition}
\begin{proof}
Since $\mathcal M_Q \supseteq L(\mathbb Z_2^\infty)$, it is invariant under conjugation by elements of $\mathbb{Z}_2^\infty$. Observe that $L(\mathbb Z_2^\infty)$ is clearly invariant under conjugation by elements of $S_\infty$. Moreover, for any $g,s\in S_\infty$ one has 
$$g(s\cdot Q^{\suppg(s)})g^{-1}=gsg^{-1}\cdot Q^{g\cdot\suppg(s)}=gsg^{-1}\cdot Q^{\suppg(gsg^{-1})}.$$ 
It follows that $\mathcal M_Q$ is $G$-invariant.

Observe that for any $s,g\in S_\infty$ one has: 
$$sQ^{\suppg(s)}gQ^{\suppg(g)}= sgQ^{g^{-1}(\suppg(s))\cup \suppg(g)}=sgQ^{\suppg(sg)}C$$ 
for some $C\in L(\mathbb{Z}_2^\infty)$, since $g^{-1}(\suppg(s))\cup \suppg(g)\supseteq \suppg(sg)$. As in addition we have $(sQ^{\suppg(s)})^* = Q^{\suppg(s)} s^{-1} = s^{-1} Q^{s^{-1} \suppg(s)} = s^{-1} Q^{\suppg(s)} = s^{-1} Q^{ \suppg(s^{-1})}$, 
it follows that linear combinations of elements of the form $sQ^{\suppg(s)}a,s\in S_\infty,a\in \mathbb Z_2^\infty$, are dense in $\mathcal M_Q$.
	
Set $\tilde{Q} = Q^{\{1,2\}}$. Consider the element $A=(12)\cdot \tilde{Q}^\perp\in L(G)$. We shall prove that $\tau(A^*B)=0$ for any $B\in \mathcal M_Q$.  For any $s\in S_\infty,a\in \mathbb Z_2^\infty$ one has:
$$\tau(A^*sQ^{\suppg(s)}a)=\tau(aQ^{\suppg(s)}\tilde{Q}^{\perp})(12)s)=0.$$ 
Indeed, if $s=(12)$ then $Q^{\suppg(s)}\tilde{Q}^{\perp}=0$. Otherwise, $aQ^{\suppg(s)}\tilde{Q}^{\perp}(12)s$ is a linear combination of the elements of the form $b(12)s,$ $b\in \mathbb Z_2^\infty$, which are not equal to the group unit and thus have zero trace.

This implies that in the decomposition 
\[ (12) = (12) \tilde{Q} + (12) \tilde{Q}^\perp\]
the first element belongs to $\mathcal{M}_Q$, and the second is orthogonal to $\mathcal{M}_Q$. Thus
$E((12)) = (12) \tilde{Q}$ and $\mathcal{M}_Q$ cannot come from a normal subgroup of $G$.
\end{proof}	

Let  now $\mathrm{Part}$ denote the collection of all partitions of $\mathbb N$ into finite sets having only finitely many parts with more than one element. 

\begin{proposition}\label{prop: example P not projection} Let $P_1\in L(\mathbb Z_2)$ be a non-zero  orthogonal projection in $L(\mathbb Z_2)$, $P_2:=P_1^\perp$. Let $\mathcal C\subseteq L(G)$ be the collection of operators of the form $$s\cdot\prod\limits_{K\in\mathcal K}(P_1^K+\sign(s|_K)\cdot P_2^K),\;\;\text{where}\;\;s\in S_\infty,\mathcal K\in\mathrm{Part}, s(K)=K\;\;\text{for all}\;\;K\in\mathcal K.$$ 
Then $\mathcal{M}_{Part}:=\overline{\mathrm{Lin}(\mathcal C)}^{w^*}$ is an exotic invariant von Neumann subalgebra of $L(G)$. 
\end{proposition}

 \begin{proof} 
 	Fix $s\in S_\infty,\mathcal K\in\mathrm{Part}, s(\mathcal K)=\mathcal K$. First, notice that for any $g\in S_\infty$ one has $$g\cdot s\cdot\prod\limits_{K\in\mathcal K}(P_1^K+\sign(s|_K)\cdot P_2^K)\cdot g^{-1}=gsg^{-1}\cdot \prod\limits_{K\in\mathcal K}(P_1^{g(K)}+\sign(s|_K)\cdot P_2^{g(K)})\in\mathcal C,$$ since $gsg^{-1}(g(K))=gs(K)=g(K)$ and $\sign(s|_{K})=\sign(gsg^{-1}|_{g(K)}$ for any $K\in\mathcal K$. It follows that $\mathcal C$ is $S_\infty$-invariant.

Now, take any $z\in\mathbb Z_2$ and $i\in\mathbb N$. Observe that 
 \begin{equation}\label{eq: P_lz}
 	P_lz=zP_l=\omega(z)P_l,\;\;l=1,2,
 \end{equation}
 where $\omega(z)=1$ if $z=e$ and $\omega(z)=-1$ if $z\neq e$. Let $K_0\in\mathcal K$ be the element of the partition $\mathcal K$ containing $i$. Then one has:
 \begin{align*}
 	z^{(i)}\cdot s\cdot\prod\limits_{K\in\mathcal K}(P_1^K+\sign(s|_K)\cdot P_2^K)\cdot (z^{-1})^{(i)}&=
 	\\  s \cdot z^{(s^{-1}(i))}\cdot (\omega_1(z^{-1})P_1^{K_0}+\sign(s|_{K_0})&\omega_2(z^{-1})P_2^{K_0})\cdot \prod\limits_{K\in\mathcal K\setminus\{K_0\}}(P_1^K+\sign(s|_K)\cdot P_2^K) =\\&s\cdot\prod\limits_{K\in\mathcal K}(P_1^K+\sign(s|_K)\cdot P_2^K),
 \end{align*} since $s^{-1}(i)\in K_0$ and $\omega(z^{-1})=\omega(z)$. Therefore, $\mathcal C$ is $Z_2^\infty$-invariant. 
 
 Let us show that $\overline{\Lin(\mathcal C)}^{w^*}$ is an algebra. Consider two elements of $\mathcal C$ defined by $s\in S_\infty$, $\mathcal K\in\Part$ and $g\in S_\infty$, $\mathcal J\in\Part$, correspondingly. It is sufficient to show that their product belongs to $\mathcal C$. One has:
 \begin{align}\label{eq: product in mathcal C}\begin{split}s\cdot\prod\limits_{K\in\mathcal K}(P_1^K+\sign(s|_K)\cdot P_2^K)\cdot g\cdot\prod\limits_{J\in\mathcal J}(P_1^J+\sign(g|_J)\cdot P_2^J)=\\ sg\cdot\prod\limits_{K\in\mathcal K}(P_1^{g^{-1}(K)}+\sign(s|_K)\cdot P_2^{g^{-1}(K)})\cdot \prod\limits_{J\in\mathcal J}(P_1^J+\sign(g|_J)\cdot P_2^J).\end{split}
 \end{align} Consider the order $\prec$ on $\Part$, where $\mathcal K_1\prec\mathcal K_2$ iff $\mathcal K_1$ refines $\mathcal K_2$, \ie for any $K\in\mathcal K_1$ there exists $J\in\mathcal K_2$ such that $K\subset J$. Let $\mathcal I$ be the minimal partition with respect to $\prec$ such that $\mathcal K\prec \mathcal I$ and $\mathcal J\prec\mathcal I$. Then $s(I)=I$ and $g(I)=I$, and therefore $sg(I)=I$  for every $I\in\mathcal I$.
 
 Further, since $g(J)=J$ for every $J\in\mathcal J$ we have:
 \begin{align*}\prod\limits_{K\in\mathcal K}(P_1^{g^{-1}(K)}+\sign(s|_K)\cdot P_2^{g^{-1}(K)})\cdot \prod\limits_{J\in\mathcal J}(P_1^J+\sign(g|_J)\cdot P_2^J)=\\g^{-1}\cdot\prod\limits_{K\in\mathcal K}(P_1^{K}+\sign(s|_K)\cdot P_2^{K})\cdot \prod\limits_{J\in\mathcal J}(P_1^J+\sign(g|_J)\cdot P_2^J)\cdot g=\\
 	g^{-1}\cdot\prod\limits_{I\in\mathcal I}(P_1^I+\sign(sg|_I)\cdot P_2^I)\cdot g=\prod\limits_{I\in\mathcal I}(P_1^I+\sign(sg|_I)\cdot P_2^I).\end{align*} Combining this with \eqref{eq: product in mathcal C} we obtain that $\mathcal C\cdot\mathcal C=\mathcal C$, which implies that $\overline{\Lin(\mathcal C)}^{w^*}$ is a von Neumann subalgebra of $L(G)$.

 Now, let $E:L(G)\to \overline{\Lin(\mathcal C)}^{w^*}$ be the $\tau$-preserving conditional expectation. By Proposition \ref{prop: E(12)}, one has $E((12))=(12)P$ for some self-adjoint element $P\in L(\mathbb Z_2^2)$. 
From the definition of $\mathcal C$ it follows that for any $A\in \Lin\{P_i^{(1)}P_j^{(2)}:(i,j)\notin\{(1,1),(2,2)\}\}$ one has $\tau((12)A\cdot\mathcal C)=\{0\}$. Moreover, $\tau((12)(P_1^{\{1,2\}}+ P_2^{\{1,2\}})\mathcal C)=\{0\}$ and
therefore, arguing as in the last proposition,  we get $P=0$.

Suppose now that $\overline{\Lin(\mathcal C)}^{w^*}$ is of the form $L(N)$ for some normal subgroup $N$ of $G$. Since $\mathcal C$ contains $(12)\cdot (P_1^{\{1,2\}}-P_2^{\{1,2\}})$, the subgroup $N$ would contain all elements from the Fourier expansion of $(12)\cdot (P_1^{\{1,2\}}-P_2^{\{1,2\}})$, \ie $(12)\cdot z$ for each $z\in \Z_2^2$. In particular, we would have $(12)\in N$. By normality of $N$, this would imply that $N=G$, which cannot hold by the last paragraph.
 \end{proof}
 
 Observe that the algebra $\mathcal{M}_{Part}$ from Proposition \ref{prop: example P not projection} is not a factor, since for any finite subset $K\subseteq\mathbb N$ the operator $P_1^K+P_2^K$ belongs to its center. 
 
 \begin{question}
 	Does $L(G)$ admit an invariant subalgebra for which $P=(12)E((12))$ is not a projection?
 \end{question}

\begin{remark}
Most of the results of this section remain true almost verbatim for the group $S_\infty \ltimes \Z_m^\infty$ with $m \in \N$, $m >2$. In particular, a construction of $\mathcal M_{Part}$ yields in that case an example of an invariant subalgebra with the corresponding operator $P$ not being a projection.
\end{remark}

\subsection*{Acknowledgements} T.A.\ would like to thank Simeng Wang for his invitation to visit the Institute of Advanced Studies in Mathematics, Harbin Institute of Technology, in May, during which time a part of this work was done. 
A.D.\ acknowledges the funding by the Long-term program of support of
the Ukrainian research teams at the Polish Academy of Sciences carried out in collaboration
with the U.S.\ National Academy of Sciences with the financial support of external partners.
Y.J.\ is partially supported by National Natural Science Foundation of China (Grant
No. 12471118). 

\bibliography{name}
\bibliographystyle{amsalpha}

\end{document}